\documentclass[10pt,reqno]{amsart}

\pdftrailerid{}
\usepackage[T1]{fontenc}
\usepackage{lmodern}
\usepackage[margin=1.08in]{geometry}
\usepackage{microtype}
\usepackage{amsmath,amssymb,mathtools,mathrsfs}
\usepackage{enumitem}
\usepackage[hidelinks]{hyperref}
\hypersetup{
  pdftitle={Sharp Deficiency Bounds for Meromorphic Functions in the Unit Disc},
  pdfauthor={Sina Nadi},
  pdfsubject={Nevanlinna Theory in the Unit Disc},
  pdfcreator={},
  pdfproducer={},
  pdfkeywords={},
}

\setlist[enumerate]{label=(\roman*),leftmargin=2.15em,itemsep=0.15em,topsep=0.35em}
\setlist[itemize]{leftmargin=1.8em,itemsep=0.15em,topsep=0.35em}
\numberwithin{equation}{section}

\newtheorem{theorem}{Theorem}[section]
\newtheorem{lemma}[theorem]{Lemma}
\newtheorem{proposition}[theorem]{Proposition}
\newtheorem{corollary}[theorem]{Corollary}

\newcommand{\C}{\mathbb C}
\newcommand{\D}{\mathbb D}
\newcommand{\Pone}{\mathbb P^1}
\newcommand{\Nbar}{\overline N}
\newcommand{\dd}{\,\mathrm d}
\newcommand{\Real}[1]{\mathit{Re}\{#1\}}
\newcommand{\Imag}[1]{\mathit{Im}\{#1\}}

\newcommand{\Ram}{\mathrm{ram}}
\newcommand{\eps}{\varepsilon}
\newcommand{\lemref}[1]{\hyperref[#1]{Lemma~\ref*{#1}}}
\newcommand{\propref}[1]{\hyperref[#1]{Proposition~\ref*{#1}}}
\newcommand{\thmref}[1]{\hyperref[#1]{Theorem~\ref*{#1}}}
\newcommand{\corref}[1]{\hyperref[#1]{Corollary~\ref*{#1}}}
\newcommand{\secref}[1]{\hyperref[#1]{Section~\ref*{#1}}}

\makeatletter
\def\@secnumfont{\bfseries}
\def\section{\@startsection{section}{1}%
  \z@{.7\linespacing\@plus\linespacing}{.5\linespacing}%
  {\normalfont\bfseries\centering}}
\makeatother

\title[Sharp Deficiency Bounds in the Unit Disc]
{Sharp Deficiency Bounds for Meromorphic Functions in the Unit Disc}
\author[Sina Nadi]{Sina Nadi}
\date{}

\begin{document}

\begin{abstract}
In 1986, Shea and Sons obtained the following bound for a meromorphic
function $f$ of finite order $\rho$ in the unit disc, under the hypothesis
$0<\lambda(f)\leq+\infty$, and for every positive integer $n$:
\[
 \sum_{a\ne\infty}\delta(a,f)
 \leq \delta(0,f^{(n)})\bigl(1+n k(f)\bigr)
       +\frac{2n(\rho+1)}{\lambda(f)}.
\]
Under the condition $0<\alpha(f)\leq+\infty$, they also obtained
\[
 \sum_{a\ne\infty}\delta(a,f)
 \leq \Delta(0,f^{(n)})\bigl(1+n k(f)\bigr)
       +\frac{2n}{\alpha(f)}.
\]
Shea and Sons asked whether the factor $2$ could be eliminated.  We prove that it can.  In fact,
whenever $0<\lambda(f)\leq+\infty$, one has
\[
 \sum_{a\ne\infty}\delta(a,f)
 \leq \delta(0,f^{(n)})\bigl(1+n k(f)\bigr)
       +\frac{n(\rho+1)}{\lambda(f)},
\]
and under the condition $0<\alpha(f)\leq+\infty$ one has
\[
 \sum_{a\ne\infty}\delta(a,f)
 \leq \Delta(0,f^{(n)})\bigl(1+n k(f)\bigr)
       +\frac{n}{\alpha(f)}.
\]
The coefficient $n$ in each of $n(\rho+1)/\lambda(f)$ and $n/\alpha(f)$ is best possible for every $n$, and the dependence on $\rho$ and $\alpha(f)$ in
these terms is also essential.
At the end, we extend both estimates to finite subsets $A$ of an
$n$-dimensional complex vector space $V\subset\C(z)$, with
$f^{(n)}$ replaced by the monic differential operator $D_Vf$ of order $n$
whose kernel is $V$.  We prove that both extensions are sharp for every $V$.
\end{abstract}

\maketitle

\section{Introduction}\label{sec:introduction}

The relation between the value distribution of a meromorphic function and the zeros of its derivatives is a classical theme in Nevanlinna theory.  The present paper concerns deficiency sums of meromorphic functions in the unit disc and their relation to deficiencies of differential operators.

Let
\[
 \D=\{z\in\C:|z|<1\},
 \qquad
 \Pone=\C\cup\{\infty\},
\]
where $\Pone$ is the Riemann sphere.  We use the standard notation $T(r,f)$,
$m(r,a,f)$, $N(r,a,f)$, and $\Nbar(r,a,f)$ of Nevanlinna theory.
See, for example, \cite{CherryYe2001}.  For real $x$, we use the convention
$\log^+x=\log\max\{1,x\}$.  The order of a meromorphic function $f$ in
$\D$ is
\[
 \rho(f)=\limsup_{r\to1^-}
 \frac{\log^+T(r,f)}{\log(1/(1-r))}.
\]
If $T(r,f)>0$ for all sufficiently large $r$, its Nevanlinna deficiency at
$a\in\Pone$ is
\[
 \delta(a,f)=\liminf_{r\to1^-}\frac{m(r,a,f)}{T(r,f)}.
\]
For $a\in\Pone$, the First Main Theorem gives
\[
 T(r,f)=m(r,a,f)+N(r,a,f)+O(1),
\]
where the bounded term is independent of $r$.  Consequently, when
$T(r,f)\to\infty$,
\[
 \delta(a,f)=1-\limsup_{r\to1^-}\frac{N(r,a,f)}{T(r,f)}.
\]
When $T(r,f)>0$ for all sufficiently large $r$, we write
\[
 k(f)=\limsup_{r\to1^-}
 \frac{\Nbar(r,\infty,f)}{T(r,f)+1},
 \qquad
 \lambda(f)=\liminf_{r\to1^-}
 \frac{T(r,f)}{\log(1/(1-r))},
 \qquad
 \alpha(f)=\limsup_{r\to1^-}
 \frac{T(r,f)}{\log(1/(1-r))}.
\]
If $T(r,f)>0$ for all sufficiently large $r$, its Valiron deficiency at
$a\in\Pone$ is
\[
 \Delta(a,f)=\limsup_{r\to1^-}\frac{m(r,a,f)}{T(r,f)}.
\]
We use the conventions $1/0=+\infty$ and $1/(+\infty)=0$.

A classical theorem of Ullrich states that if $f$ is a non-affine entire
function of finite order in $\C$, then
\[
 \sum_{a\in\C}\delta(a,f)\leq\delta(0,f').
\]
See \cite{Ullrich1929}.  The results below concern the corresponding
deficiency sum for meromorphic functions in $\D$.

In 1983, Sons recorded the following result
\cite[Theorem~2.14]{Sons1983}.  If $f$ is meromorphic in $\D$ and has
finite order and positive lower order, then, for every positive integer $n$,
\begin{equation}\label{eq:Sons-derivative}
 \frac{1}{n+1}\sum_{a\ne\infty}\delta(a,f)
 \leq \delta(0,f^{(n)}).
\end{equation}
If $f$ is analytic in $\D$, then the factor $1/(n+1)$ may be omitted.

In 1986, Shea and Sons proved that a meromorphic function of finite order
$\rho$ in the unit disc with $0<\lambda(f)\leq+\infty$ satisfies
\begin{equation}\label{eq:SheaSons}
 \sum_{a\ne\infty}\delta(a,f)
 \leq
 \delta(0,f')\bigl(1+k(f)\bigr)
 +\frac{2(\rho+1)}{\lambda(f)}.
\end{equation}
For every positive integer $n$, they further proved
\begin{equation}\label{eq:SheaSons-43}
 \sum_{a\ne\infty}\delta(a,f)
 \leq
 \delta(0,f^{(n)})\bigl(1+n k(f)\bigr)
 +\frac{2n(\rho+1)}{\lambda(f)}.
\end{equation}
For a meromorphic function with $0<\alpha(f)\leq+\infty$, they also proved
\begin{equation}\label{eq:SheaSons-42}
 \sum_{a\ne\infty}\delta(a,f)
 \leq
 \Delta(0,f^{(n)})\bigl(1+n k(f)\bigr)
 +\frac{2n}{\alpha(f)}.
\end{equation}
See \cite{SheaSons1986}.  In 1989, Brannan and Hayman recorded the following question of Sons as Problem~1.40 in their collection of research problems in complex analysis \cite{BrannanHayman1989}.  It was later reproduced as Problem~1.40 in the fiftieth-anniversary edition of \emph{Research Problems in Function Theory} \cite{HaymanLingham2019}.

\medskip
\noindent\textbf{Sons' problem.}
Can the factor $2$ in \eqref{eq:SheaSons} be eliminated?

\medskip
The problem states that an affirmative answer would be best possible.

Our first result answers this question.

\begin{theorem}\label{thm:main}
Let $f$ be meromorphic in $\D$ and of finite order $\rho$.  Suppose that
$0<\lambda(f)\leq+\infty$.
\begin{enumerate}[label=\textup{(\alph*)}]
\item One has
\begin{equation}\label{eq:main}
 \sum_{a\ne\infty}\delta(a,f)
 \leq
 \delta(0,f')\bigl(1+k(f)\bigr)
 +\frac{\rho+1}{\lambda(f)}.
\end{equation}
\item The coefficient $1$ of $(\rho+1)/\lambda(f)$ in
\eqref{eq:main} is best possible.
\end{enumerate}
\end{theorem}

The sum in \eqref{eq:main} is the supremum of the sums over finite subsets of
$\C$.

The dependence on $\rho$ is also essential.  \propref{prop:positive-order}
shows that $(\rho+1)/\lambda(f)$ cannot in general be replaced by
$1/\lambda(f)$.

When $\lambda(f)=0$, the term $(\rho+1)/\lambda(f)$ is $+\infty$ and the
formula cannot give any information.  We therefore state
\thmref{thm:main} only in the positive-$\lambda$ regime.  This also avoids undefined deficiency quotients for derivatives of constant or affine
functions.

The extension of defect relations and Second Main Theorems from fixed values to rational or small meromorphic functions has a long history.  Frank and Weissenborn proved a defect relation for rational functions as targets in the plane \cite{FrankWeissenborn1986}.  Steinmetz proved a weak form of the Second Main Theorem for small functions in the plane, without the ramification term,
using a Wronskian differential polynomial \cite{Steinmetz1986}.  Yamanoi
subsequently established the corresponding theorem with counting functions truncated at level one for small functions \cite{Yamanoi2004} and, in a separate paper, for rational functions as targets \cite{Yamanoi2005}.  In 2016, Ciechanowicz proved a unit-disc analogue of Ullrich's theorem for admissible meromorphic functions and finite families of small functions
\cite[Theorem~2.8]{Ciechanowicz2016}.

Our second and third results concern finite-dimensional subspaces of $\C(z)$.  For an
$n$-dimensional complex vector space $V\subset\C(z)$ with $n\geq1$,
\secref{sec:rational-targets} defines the monic differential operator $D_V$ of
order $n$ whose kernel is $V$, as well as the deficiency $\delta(a,f)$ for
$a\in V$.

\begin{theorem}\label{thm:rational-target-extension}
Let $f$ be meromorphic in $\D$, of finite order $\rho$, and suppose that
$0<\lambda(f)\leq+\infty$.  If $V\subset\C(z)$ is an $n$-dimensional complex vector space with $n\geq1$ and $T(r,D_Vf)>0$ for all sufficiently large
$r$, then, for every finite subset $A\subset V$,
\begin{equation}\label{eq:rational-target-extension}
 \sum_{a\in A}\delta(a,f)
 \leq
 \delta(0,D_Vf)\bigl(1+n k(f)\bigr)
 +\frac{n(\rho+1)}{\lambda(f)}.
\end{equation}
\end{theorem}

For every $n\geq1$ and every $n$-dimensional subspace $V\subset\C(z)$,
equality in \eqref{eq:rational-target-extension} can be attained.  Both terms on the right can be strictly positive,
and the dependence on $\rho$ cannot be omitted.  See
\propref{prop:all-V-nevanlinna-sharpness}.

The assumption $T(r,D_Vf)>0$ for all sufficiently large $r$ is imposed only to ensure that $\delta(0,D_Vf)$ is well defined.

\begin{theorem}\label{thm:valiron-rational-target-extension}
Let $f$ be meromorphic in $\D$ and suppose that
$0<\alpha(f)\leq+\infty$.  If $V\subset\C(z)$ is an $n$-dimensional complex vector space with $n\geq1$ and $T(r,D_Vf)>0$ for all sufficiently large
$r$, then, for every finite subset $A\subset V$,
\begin{equation}\label{eq:valiron-rational-target-extension}
 \sum_{a\in A}\delta(a,f)
 \leq
 \Delta(0,D_Vf)\bigl(1+n k(f)\bigr)
 +\frac{n}{\alpha(f)}.
\end{equation}
\end{theorem}

For every $n\geq1$ and every $n$-dimensional subspace $V\subset\C(z)$,
equality in \eqref{eq:valiron-rational-target-extension} can be attained.
Both terms on the right can be strictly positive, and the term
$n/\alpha(f)$ cannot be omitted.  See
\propref{prop:all-V-valiron-sharpness}.

The positivity assumption on $T(r,D_Vf)$ in this theorem serves only to make
$\Delta(0,D_Vf)$ well defined.

For constant targets and $D_V=D^n$, Theorems~\ref{thm:rational-target-extension}
and~\ref{thm:valiron-rational-target-extension} improve
\eqref{eq:SheaSons-43} and \eqref{eq:SheaSons-42}, respectively, by removing the factor $2$.

\begin{corollary}\label{cor:improved-SheaSons-43}
Let $f$ be meromorphic in $\D$, of finite order $\rho$, and suppose that
$0<\lambda(f)\leq+\infty$.  Then, for every positive integer $n$,
\begin{equation}\label{eq:improved-SheaSons-43}
 \sum_{a\ne\infty}\delta(a,f)
 \leq
 \delta(0,f^{(n)})\bigl(1+n k(f)\bigr)
 +\frac{n(\rho+1)}{\lambda(f)}.
\end{equation}
The estimate is sharp for every $n$, and its dependence on $\rho$ cannot be omitted.
\end{corollary}

\begin{proof}
Take $V=\operatorname{span}_{\C}\{1,z,\ldots,z^{n-1}\}$ in
\thmref{thm:rational-target-extension}.  Then
$D_V=D^n=d^n/dz^n$.  For $n=1$, this is $D=d/dz$.
The condition $\lambda(f)>0$ implies that $T(r,f)\to+\infty$.  Applying the outer-annulus argument from \secref{sec:main-proof} to $f^{(n)}$ and then integrating $n$ times shows that $T(r,f^{(n)})>0$ for all sufficiently large
$r$.  We may therefore apply
\eqref{eq:rational-target-extension} to finite sets of constant functions and
take the supremum over all such sets.
Proposition~\ref{prop:nth-order-sharpness} proves both sharpness assertions.
Thus \eqref{eq:improved-SheaSons-43} and both additional assertions follow.
\end{proof}

\begin{corollary}\label{cor:improved-SheaSons-42}
Let $f$ be meromorphic in $\D$ and suppose that
$0<\alpha(f)\leq+\infty$.  Then, for every positive integer $n$,
\begin{equation}\label{eq:improved-SheaSons-42}
 \sum_{a\ne\infty}\delta(a,f)
 \leq
 \Delta(0,f^{(n)})\bigl(1+n k(f)\bigr)
 +\frac{n}{\alpha(f)}.
\end{equation}
The estimate is sharp for every $n$, and its dependence on $\alpha(f)$ cannot be omitted.
\end{corollary}

\begin{proof}
As in the preceding proof, take
$V=\operatorname{span}_{\C}\{1,z,\ldots,z^{n-1}\}$.  Since
$\alpha(f)>0$ implies $T(r,f)\to+\infty$, the same outer-annulus argument
makes $T(r,f^{(n)})>0$ for all sufficiently large $r$.  Apply
\thmref{thm:valiron-rational-target-extension} to finite sets of constant functions and take the supremum.
Proposition~\ref{prop:valiron-sharpness} proves both sharpness assertions.
Thus \eqref{eq:improved-SheaSons-42} and both additional assertions follow.
\end{proof}

\medskip

We nevertheless prove \thmref{thm:main} first because it is the problem under consideration and its proof isolates the first-order argument before the Wronskian and higher-order estimates required in
\secref{sec:rational-targets} are introduced.  Theorem~\ref{thm:main} is
proved in \secref{sec:main-proof}, and
Theorem~\ref{thm:rational-target-extension} in
\secref{sec:rational-targets}.  The proof of
Theorem~\ref{thm:valiron-rational-target-extension} is given in
\secref{sec:valiron-proof}.
The sharpness assertions for arbitrary rational kernels are proved in
Appendix~\ref{sec:kernelwise-sharpness}.

\section*{Acknowledgments}
We are deeply grateful to Alexandre Eremenko for the time and care
he devoted to reading the manuscript and for his valuable comments
and suggestions. We also thank David Drasin for his helpful feedback.

\section{Proof of Theorem~\ref{thm:main}}\label{sec:main-proof}

The hypothesis $\lambda(f)>0$ implies that $T(r,f)\to\infty$ as $r\to1^-$.
It also implies that $f'$ is not constant.  Indeed, a function with constant derivative is affine and has bounded characteristic in $\D$.  We claim that
$T(r,f')>0$ for all sufficiently large $r$.  Suppose otherwise.  Then there is a sequence $r_j\to1^-$ such that $T(r_j,f')\leq0$.  A pole of $f'$ away from
the origin would contribute a positive amount to $N(r_j,\infty,f')$ for all sufficiently large $j$, whereas the possible contribution at the origin tends to zero.  Hence the origin is the only possible pole of $f'$.  Let $s\geq0$ be
its order, with $s=0$ if $f'$ is analytic at the origin.  Then
\[
 N(r_j,\infty,f')=s\log r_j,
 \qquad
 0\leq m(r_j,\infty,f')\leq-s\log r_j\longrightarrow0.
\]
The function $G(z)=z^sf'(z)$ is analytic in $\D$, and
$m(r_j,\infty,G)\leq m(r_j,\infty,f')\to0$.  Since the circular means of the
subharmonic function $\log^+|G|$ are nondecreasing, it follows that
$|G|\leq1$ in $\D$.  Thus $f'$ is bounded on $\{1/2\leq|z|<1\}$.
Integrating along paths of uniformly bounded length in this annulus shows that
$f$ is bounded there.  Since $f$ has only finitely many poles in
$\{|z|<1/2\}$, it follows that $T(r,f)=O(1)$ as $r\to1^-$, contradicting
$\lambda(f)>0$.

Fix a finite set of distinct finite values
\[
 A=\{a_1,\ldots,a_q\}\subset\C
\]
and put
\[
 M_A(r)=\sum_{j=1}^q m(r,a_j,f).
\]
If $q<2$, we add auxiliary finite values.  An upper estimate for the enlarged
sum also gives an upper estimate for $M_A$.

Let
\[
 m_1(r)=m(r,0,f'),
 \qquad N_1(r)=N(r,0,f').
\]
The ramification counting function satisfies
\begin{equation}\label{eq:ramification}
 N_{\Ram}(r,f)
 =N_1(r)+N(r,\infty,f)-\Nbar(r,\infty,f).
\end{equation}
Indeed, a pole of $f$ of order $p$ contributes $p-1$ to the ramification
divisor.

\begin{lemma}\label{lem:interpolation}
Let
\[
 U(r)=T(r,f')
\]
and, for all sufficiently large $r$,
\[
 \alpha(r)=\min\left\{\frac{m_1(r)}{U(r)},1\right\}.
\]
Let $E$ be a real-valued function defined for all sufficiently large $r$, and
suppose that
\begin{align}
 M_A(r)&\leq m_1(r)+E(r),\label{eq:first-general}\\
 M_A(r)&\leq T(r,f)+\Nbar(r,\infty,f)-N_1(r)+E(r).
 \label{eq:second-general}
\end{align}
Then
\begin{equation}\label{eq:interpolated}
 M_A(r)
 \leq
 \alpha(r)\bigl(T(r,f)+\Nbar(r,\infty,f)\bigr)
 +E(r)+O_{f'}(1).
\end{equation}
\end{lemma}

\begin{proof}
Multiply \eqref{eq:first-general} by $1-\alpha(r)$ and
\eqref{eq:second-general} by $\alpha(r)$, and add.  It remains to estimate
\[
 R(r)=m_1(r)-\alpha(r)\bigl(m_1(r)+N_1(r)\bigr).
\]
Jensen's formula gives a constant $c_{f'}$ such that
\begin{equation}\label{eq:Jensen-constant}
 m_1(r)+N_1(r)=U(r)+c_{f'}.
\end{equation}
If $m_1(r)\leq U(r)$, then
\[
 R(r)=-c_{f'}\frac{m_1(r)}{U(r)},
\]
so $R(r)$ is bounded above.  If $m_1(r)>U(r)$, then $\alpha(r)=1$ and
\[
 R(r)=-N_1(r),
\]
which is bounded above for all sufficiently large $r$, since $N_1(r)$ is
nondecreasing.  This proves
\eqref{eq:interpolated}.
\end{proof}

The truncation in the definition of $\alpha$ is needed only to ensure that
$0\leq\alpha(r)\leq1$.  Indeed, \eqref{eq:Jensen-constant} gives
\[
 \frac{m_1(r)}{U(r)}
 =1-\frac{N_1(r)+c_{f'}}{U(r)}.
\]
Since $c_{f'}$ may be negative, the untruncated quotient can exceed $1$.
The truncation therefore guarantees the convex combination required in the proof.

We next state the two external estimates used in the proof.

\begin{theorem}[{\cite[Theorem~3.2.2]{CherryYe2001}}]
\label{thm:GG-source}
Let $0<R\leq+\infty$, let $g\not\equiv0$ be meromorphic in
$D(R)=\{z\in\C:|z|<R\}$, and let $0<\gamma<1$.  Fix $r_0\in(0,R)$.  Write
\[
 g(z)=c z^m(1+O(z))\qquad(z\to0),
\]
where $m=\operatorname{ord}_0g\in\mathbb Z$ and $c\ne0$, and put
\[
 \beta_1(g,r_0)
 =|m|\log^+\frac1{r_0}+\bigl|\log|c|\bigr|+\log2.
\]
Then, for $r_0<r<p<R$,
\begin{align}
 \frac1{2\pi}\int_0^{2\pi}
 \left|\frac{g'(re^{i\theta})}{g(re^{i\theta})}\right|^\gamma
 \dd\theta
 &\leq C_{\rm gg}(\gamma)
 \left(\frac{p}{r(p-r)}\right)^\gamma
 \bigl(2T(p,g)+\beta_1(g,r_0)\bigr)^\gamma,
 \label{eq:GG-Lgamma}
\end{align}
where
\[
 C_{\rm gg}(\gamma)
 =2^\gamma+\bigl(8+2\cdot2^\gamma\bigr)
 \sec\frac{\gamma\pi}{2}.
\]
\end{theorem}

\begin{corollary}\label{cor:GG-log}
Let $g\not\equiv0$ be meromorphic in $\D$, and fix $r_0\in(0,1)$.
Then, for $r_0<r<p<1$,
\begin{equation}\label{eq:GG}
 m\!\left(r,\frac{g'}g\right)
 \leq
 \log^+T(p,g)+\log\frac{p}{r(p-r)}+O_{g,r_0}(1),
\end{equation}
where the bounded term is independent of $r$ and $p$.
\end{corollary}

\begin{proof}
Set $\gamma=1/2$.  Since
$\log^+x\leq\gamma^{-1}\log(1+x^\gamma)$, Jensen's inequality and
\thmref{thm:GG-source} give
\[
 m\!\left(r,\frac{g'}g\right)
 \leq \frac1\gamma\log\left(
 1+\frac1{2\pi}\int_0^{2\pi}
 \left|\frac{g'(re^{i\theta})}{g(re^{i\theta})}\right|^\gamma
 \dd\theta\right),
\]
and \eqref{eq:GG} follows from \eqref{eq:GG-Lgamma}.
\end{proof}

\begin{theorem}[{\cite[Theorem~4.2.1\textup{(1)}]{CherryYe2001}}]
\label{thm:two-radius-SMT}
Let $b_1,\ldots,b_s\in\Pone$ be distinct, with at least two of them
finite, and let $g$ be a nonconstant meromorphic function in $\D$.
Fix $r_0\in(0,1)$ such that $T(r_0,g)>0$.  There is a constant
$C=C(g,b_1,\ldots,b_s,r_0)$ such that, for $r_0<r<p<1$,
\begin{align}
 (s-2)T(r,g)-\sum_{j=1}^sN(r,b_j,g)+N_{\Ram}(r,g)
 &\leq \log T(p,g)+\log\frac{p}{r(p-r)}+C.
 \label{eq:two-radius-SMT-counting}
\end{align}
Here $N_{\Ram}(r,g)$ is the ramification counting function of $g$.
\end{theorem}

We now begin the proof of \thmref{thm:main}.

\begin{proof}[Proof of Theorem~\ref{thm:main}\textup{(a)}]
We use estimates which hold for every sufficiently large radius.  Choose
$r_0\in(0,1)$ so that $T(r_0,f)>0$.

By \thmref{thm:two-radius-SMT} and the First Main Theorem, if
$b_1,\ldots,b_s\in\Pone$ are distinct and at least two are finite, then
\begin{align}
 \sum_{j=1}^s m(r,b_j,f)-2T(r,f)+N_{\Ram}(r,f)
 &\leq
 \log T(p,f)+\log\frac{p}{r(p-r)}+O_{f,\boldsymbol b,r_0}(1),
 \label{eq:two-radius-SMT}
\end{align}
for $r_0<r<p<1$, where $\boldsymbol b=(b_1,\ldots,b_s)$.

Set
\[
 P(w)=\prod_{j=1}^q(w-a_j),
 \qquad H=P(f),
 \qquad Q(w)=\frac{P'(w)}{P(w)}=\sum_{j=1}^q\frac1{w-a_j}.
\]
Choose pairwise disjoint discs centered at the points $a_j$.  In a
sufficiently small disc centered at $a_j$,
\[
 Q(w)=\frac1{w-a_j}+O_A(1),
\]
whereas outside their union the function
\[
 \sum_{j=1}^q\log^+\frac1{|w-a_j|}
\]
is bounded.  It follows that
\[
 \sum_{j=1}^q\log^+\frac1{|w-a_j|}
 \leq \log^+|Q(w)|+O_A(1).
\]
Since $H'/H=f'Q(f)$, integration over $|z|=r$ gives
\begin{equation}\label{eq:finite-values}
 M_A(r)
 \leq
 m(r,0,f')+m\!\left(r,\frac{H'}H\right)+O_A(1).
\end{equation}

Fix $\eps>0$ and set
\[
 p=\frac{1+r}{2}.
\]
By the definition of the order,
\begin{equation}\label{eq:order-estimate}
 \log^+T(p,f)
 \leq
 (\rho+\eps)\log\frac1{1-r}+O_{f,\eps}(1),
\end{equation}
and
\begin{equation}\label{eq:separation-estimate}
 \log\frac{p}{r(p-r)}
 =\log\frac1{1-r}+O(1).
\end{equation}
Moreover,
\[
 T(p,H)=qT(p,f)+O_A(1),
\]
so \eqref{eq:order-estimate} also holds with $H$ in place of $f$, after the bounded term is allowed to depend on $A$.
By \eqref{eq:finite-values} and \corref{cor:GG-log},
\begin{equation}\label{eq:first-estimate}
 M_A(r)
 \leq
 m_1(r)+(\rho+1+\eps)\log\frac1{1-r}
 +O_{f,A,r_0,\eps}(1).
\end{equation}

Apply \eqref{eq:two-radius-SMT} to the values
$a_1,\ldots,a_q,\infty$.  Since
\[
 m(r,\infty,f)=T(r,f)-N(r,\infty,f),
\]
we obtain from \eqref{eq:ramification}
\begin{equation}\label{eq:second-estimate}
 M_A(r)
 \leq
 T(r,f)+\Nbar(r,\infty,f)-N_1(r)
 +(\rho+1+\eps)\log\frac1{1-r}
 +O_{f,A,r_0,\eps}(1).
\end{equation}

Apply \lemref{lem:interpolation} to \eqref{eq:first-estimate} and
\eqref{eq:second-estimate}.  Thus
\begin{equation}\label{eq:main-radius-estimate}
 M_A(r)
 \leq
 \alpha(r)\bigl(T(r,f)+\Nbar(r,\infty,f)\bigr)
 +(\rho+1+\eps)\log\frac1{1-r}
 +O_{f,A,r_0,\eps}(1).
\end{equation}

Choose a sequence $r_n\to1^-$ such that
\[
 \frac{m(r_n,0,f')}{T(r_n,f')}
 \longrightarrow\delta(0,f').
\]
Then
\begin{equation}\label{eq:alpha-limit}
 \alpha(r_n)\longrightarrow\min\{\delta(0,f'),1\}
 \leq\delta(0,f').
\end{equation}
Divide \eqref{eq:main-radius-estimate} by $T(r,f)$ and let $r=r_n$.
Since $T(r,f)\to\infty$,
\[
 \limsup_{r\to1^-}
 \frac{\Nbar(r,\infty,f)}{T(r,f)}=k(f),
\]
where the equality follows from the definition with denominator $T(r,f)+1$.
The definition of $\lambda(f)$ also gives
\[
 \limsup_{r\to1^-}
 \frac{\log(1/(1-r))}{T(r,f)}
 =\frac1{\lambda(f)}.
\]
It follows that
\begin{equation}\label{eq:finite-A-bound}
 \limsup_{n\to\infty}\frac{M_A(r_n)}{T(r_n,f)}
 \leq
 \delta(0,f')\bigl(1+k(f)\bigr)
 +\frac{\rho+1+\eps}{\lambda(f)}.
\end{equation}
For the finite set $A$,
\[
 \sum_{a\in A}\delta(a,f)
 \leq
 \liminf_{r\to1^-}\frac{M_A(r)}{T(r,f)}
 \leq
 \limsup_{n\to\infty}\frac{M_A(r_n)}{T(r_n,f)}.
\]
Combining this with \eqref{eq:finite-A-bound}, letting $\eps\downarrow0$,
and taking the supremum over all finite $A\subset\C$ proves
\eqref{eq:main}.

\end{proof}

This proves part~\textup{(a)}.  We now turn to sharpness.  The construction relies on the following classical growth theorem of Tsuji for the modular covering.

\begin{theorem}[{\cite[Theorem~XI.29]{Tsuji1959}}]
\label{thm:Tsuji-modular}
Let $h\colon\D\to\Pone\setminus\{0,1,\infty\}$ be the modular
covering obtained from the classical modular function by a conformal
identification of the upper half-plane with $\D$.  Then
\[
 T(r,h)=\log\frac1{1-r}+O(1)
 \qquad\left(\frac12\leq r<1\right).
\]
\end{theorem}

\begin{proof}[Proof of Theorem~\ref{thm:main}\textup{(b)}]
Let $h$ be the modular covering in \thmref{thm:Tsuji-modular}.  Since $h$ is a
covering map, it is locally biholomorphic.  Hence $h'(z)\ne0$ for every
$z\in\D$.  Choose $c\ne0$ so that $F=ch$ satisfies $|F'(0)|=1$.
The function $F$ omits the two finite values $0$ and $c$.  Its derivative is zero-free, and
\[
 \rho(F)=0,\qquad \lambda(F)=1,\qquad k(F)=0.
\]
Since $F'$ is zero-free, Jensen's formula and the normalization
$|F'(0)|=1$ give $m(r,0,F')=T(r,F')$.  Hence $\delta(0,F')=1$.
The two omitted values and part~\textup{(a)} now give equality in
\eqref{eq:main}.  Hence the coefficient of
$(\rho+1)/\lambda(f)$ cannot be smaller than $1$.
\end{proof}

The preceding construction proves the optimality of the numerical coefficient
in \eqref{eq:main}, but it has order zero.  The following proposition, based on the construction in \cite[proof of Theorem~4]{Miles1992}, shows that the dependence on $\rho$ cannot be omitted.

\begin{proposition}\label{prop:positive-order}
For every $\rho>0$ and every integer $M>\rho+1$, there is a
meromorphic function $f$ in $\D$ such that
\[
 \rho(f)=\rho,\qquad \lambda(f)=M,\qquad k(f)=\frac1M,
 \qquad \delta(0,f)=1,
 \qquad \delta(0,f')=\frac{M-\rho-1}{M+1}.
\]
For this function,
\[
 \sum_{a\ne\infty}\delta(a,f)
 =\delta(0,f')\bigl(1+k(f)\bigr)
  +\frac{\rho+1}{\lambda(f)}=1,
\]
so equality holds in \eqref{eq:main}.  On the other hand,
\[
 \delta(0,f')\bigl(1+k(f)\bigr)+\frac1{\lambda(f)}
 =1-\frac{\rho}{M}<1,
\]
so $(\rho+1)/\lambda(f)$ cannot in general be replaced by
$1/\lambda(f)$.
\end{proposition}

\begin{proof}
Choose $s_1>0$ sufficiently small that $(s_j)$ is strictly decreasing
and $n_j\geq2$.  Put
\[
 s_{j+1}=\exp(-s_j^{-\rho}),\qquad r_j=1-s_j,
 \qquad n_j=\lfloor s_j^{-(\rho+1)}\rfloor,
 \qquad L_j=\log\frac1{s_j}.
\]
Then $s_j\to0$ and
\[
 L_{j+1}=s_j^{-\rho},\qquad
 \log n_j=(\rho+1)L_j+o(1),\qquad
 \sum_{k<j}n_k=o(n_j),\qquad j=o(L_j).
\]
Fix $a\in(0,r_1)$ and define
\[
 P(z)=\left(1-\frac za\right)
 \prod_{j=1}^{\infty}\left(1-\left(\frac z{r_j}\right)^{n_j}\right),
 \qquad f=P^{-M}.
\]
For each $R<1$, the tail of the product is dominated by
$\sum_j(2R/(1+R))^{n_j}$.  Thus the product converges locally uniformly in
$\D$.  Its zeros are simple and $P(0)=1$.

Write $N_P(r)=N(r,0,P)$.  If $r_j\leq r<r_{j+1}$, set $s=1-r$ and
$x=s/s_j$.  The exact zero count gives
\[
 N_P(r)=\log\frac ra+\sum_{k<j}n_k\log\frac r{r_k}
             +n_j\log\frac r{r_j}.
\]
As $j\to\infty$, uniformly for $r_j\leq r<r_{j+1}$,
\begin{equation}\label{eq:positive-order-growth}
 N_P(r)=L_j(1+o(1))+L_{j+1}(1-x)(1+o(1)),
 \qquad m(r,0,P)=O(j).
\end{equation}
Indeed, the $k=j-1$ term is
$n_{j-1}s_{j-1}(1+o(1))=L_j(1+o(1))$, the preceding terms are
$o(L_j)$, and
\[
 n_j\log\frac r{r_j}=L_{j+1}(1-x)(1+o(1)).
\]
For the proximity estimate, use
\[
 \frac1{2\pi}\int_0^{2\pi}\log^-|1-qe^{i\theta}|\dd\theta=O(1)
 \qquad(q\geq0).
\]
Here $\log^-x=\max\{-\log x,0\}$ for $x>0$, with
$\log^-0=+\infty$.
The factors with $k\leq j+1$ contribute $O(j)$, while those with
$k\geq j+2$ contribute $o(1)$.

Since $P(0)=1$, Jensen's formula gives
$T(r,P)=N_P(r)+m(r,0,P)$.  Moreover,
\[
 \log\frac1{1-r}=L_j-\log x,
 \qquad -\log x\leq(1+o(1))L_{j+1}(1-x)
\]
uniformly in the interval under consideration.  Together with
\eqref{eq:positive-order-growth}, these estimates give
$\lambda(P)\geq1$, while $r=r_j$ gives the reverse inequality.  Also,
\[
 T(r,P)=O((1-r)^{-\rho}),
 \qquad T(1-s_j/2,P)\geq\left(\frac12+o(1)\right)L_{j+1},
\]
so $\rho(P)=\rho$.  Hence
\[
 T(r,f)=M T(r,P),\qquad \rho(f)=\rho,
 \qquad \lambda(f)=M.
\]
Since $\Nbar(r,\infty,f)=N_P(r)$ and every pole of $f$ has multiplicity
$M$, \eqref{eq:positive-order-growth}, first for arbitrary $r$ and then at
$r=r_j$, gives $k(f)=1/M$.  The function $f$ omits $0$, and therefore
$\delta(0,f)=1$.

Fix $\eta>0$ and set
\[
 R_j=r_j e^{-\eta/n_j},\qquad A_j=N_P(R_j),
 \qquad Q=\frac{P'}P.
\]
For all sufficiently large $j$, one has $r_{j-1}<R_j<r_j$, and applying
\eqref{eq:positive-order-growth} with $j$ replaced by $j-1$ gives
$A_j=L_j(1+o(1))$.  Uniformly on $|z|=R_j$, one has
\begin{equation}\label{eq:positive-order-circle}
 \log|P(z)|=A_j+O(j),\qquad |Q(z)|\asymp n_j.
\end{equation}
For the second relation, write
\[
 zQ(z)=\frac{z}{z-a}-\sum_{k\geq1}
 \frac{n_k(z/r_k)^{n_k}}{1-(z/r_k)^{n_k}}.
\]
The term with $k=j$ has modulus comparable with $n_j$.  The sum over
$k<j$ is $o(n_j)$, and the remaining sum is $o(1)$.  For the first
relation in \eqref{eq:positive-order-circle}, factor $(z/r_k)^{n_k}$
from the terms with $k<j$.  The $j$th factor is bounded above and below,
and the product over $k>j$ is $1+o(1)$.

Now $f'=-MP^{-M}Q$.  By \eqref{eq:positive-order-circle},
\[
 \log|f'(R_je^{i\theta})|
 =-(M-\rho-1)L_j+o(L_j)<0
\]
uniformly for all sufficiently large $j$.  Hence
$m(R_j,\infty,f')=0$.  The poles of $Q$ are the simple zeros of $P$, so
$N(R_j,\infty,Q)=A_j$, while $Q(0)=-1/a$.  Jensen's formula and
\eqref{eq:positive-order-circle} therefore give
\[
 N(R_j,0,Q)=A_j+\log n_j+O(1).
\]
At each zero of $P$, the function $f'$ has a pole of order $M+1$.
Consequently,
\[
 T(R_j,f')=(M+1)A_j,
 \qquad
 m(R_j,0,f')=MA_j-\log n_j+O(1).
\]
It follows that
\[
 \delta(0,f')\leq
 \lim_{j\to\infty}\frac{m(R_j,0,f')}{T(R_j,f')}
 =\frac{M-\rho-1}{M+1}.
\]
Since $\delta(0,f)=1$, part~\textup{(a)} of \thmref{thm:main} and the
preceding estimate give
\[
 1\leq\sum_{a\ne\infty}\delta(a,f)
 \leq\delta(0,f')\left(1+\frac1M\right)+\frac{\rho+1}{M}
 \leq\frac{M-\rho-1}{M+1}\left(1+\frac1M\right)
      +\frac{\rho+1}{M}=1.
\]
Thus all these inequalities are equalities.  In particular,
\[
 \delta(0,f')=\frac{M-\rho-1}{M+1},
 \qquad
 \sum_{a\ne\infty}\delta(a,f)=1.
\]
Finally,
\[
 \delta(0,f')\bigl(1+k(f)\bigr)+\frac1{\lambda(f)}
 =1-\frac\rho M,
\]
which proves the assertion.
\end{proof}

\begin{proposition}\label{prop:nth-order-sharpness}
Let $n\geq1$, let $\rho>0$, and let $M$ be an integer satisfying
\[
 M>n(\rho+1).
\]
There is a meromorphic function $f$ in $\D$ such that
\[
 \rho(f)=\rho,
 \qquad \lambda(f)=M,
 \qquad k(f)=\frac1M,
 \qquad \delta(0,f)=1,
\]
and
\[
 \delta(0,f^{(n)})=\frac{M-n(\rho+1)}{M+n}>0.
\]
For
\[
 V=\operatorname{span}_{\C}\{1,z,\ldots,z^{n-1}\},
 \qquad A=\{0\},
\]
equality holds in \eqref{eq:rational-target-extension}.  Consequently, the coefficient $n$ of $(\rho+1)/\lambda(f)$ in that estimate is best possible for every $n$, and the dependence on $\rho$ cannot be omitted.
\end{proposition}

\begin{proof}
Use the sequences, the product $P$, and the function $f=P^{-M}$ from the
proof of \propref{prop:positive-order}.  The calculations there give
\[
 \rho(f)=\rho,
 \qquad \lambda(f)=M,
 \qquad k(f)=\frac1M,
 \qquad \delta(0,f)=1.
\]
It remains to control the terms created by the $n$th derivative.

Put $\mathcal D=w\,\dd/\dd w$ and, for $1\leq q\leq n$, define
\[
 H_{M,q}(w)=(1-w)^M\mathcal D^q(1-w)^{-M}.
\]
Since $H_{M,n}(w)=Mw+O(w^2)$ at the origin, choose $c\in(0,1)$ such that
$H_{M,n}(w)\ne0$ whenever $0<|w|\leq c$.  Set
\[
 \eta=-\log c,
 \qquad R_j=r_j e^{-\eta/n_j},
 \qquad A_j=N_P(R_j).
\]
The calculation leading to \eqref{eq:positive-order-circle} gives
\begin{equation}\label{eq:nth-sharpness-P-circle}
 A_j=L_j(1+o(1)),
 \qquad \log|P(z)|=A_j+O(j)
 \quad (|z|=R_j).
\end{equation}

We record the estimate that rules out cancellation.  Write
\[
 \mathcal E=z\frac{\dd}{\dd z},
 \qquad u_k(z)=\left(\frac z{r_k}\right)^{n_k},
 \qquad G_j(z)=\frac{P(z)}{1-u_j(z)}.
\]
For $1\leq\ell\leq n$, let
$\Phi_\ell(w)=\mathcal D^\ell\log(1-w)$.  The rational function
$\Phi_\ell$ is $O(|w|)$ for $|w|\leq1/2$ and is bounded for $|w|\geq2$.
We therefore have
\[
 \mathcal E^\ell\log G_j(z)
 =O(1)+\sum_{k\ne j}n_k^\ell\Phi_\ell(u_k(z))
 =o(n_j^\ell)
\]
uniformly on $|z|=R_j$.
Indeed, the terms with $k<j$ are bounded by
$O(\sum_{k<j}n_k^\ell)=o(n_j^\ell)$.  For $k>j$,
\[
 |u_k(z)|\leq\exp\left(-\frac{s_jn_k}{2}\right).
\]
The facts that $n_{k+1}\geq2n_k$ for all sufficiently large $k$ and
$s_jn_{j+1}/\log n_{j+1}\to+\infty$ show that
\[
 \sum_{k>j}n_k^\ell
 \exp\left(-\frac{s_jn_k}{2}\right)
 \leq\sum_{m\geq0}(2^mn_{j+1})^\ell
 \exp\left(-\frac{s_j2^mn_{j+1}}2\right)=o(1).
\]
Here $x^\ell e^{-s_jx/2}$ is decreasing for $x\geq n_{j+1}$ when $j$ is
sufficiently large.

Let $B_j=G_j^{-M}$.  Each quotient
$\mathcal E^\ell B_j/B_j$ is a differential polynomial in
$\mathcal E\log B_j,\ldots,\mathcal E^\ell\log B_j$.  Hence
\[
 \frac{\mathcal E^\ell B_j}{B_j}=o(n_j^\ell)
 \qquad(1\leq\ell\leq n).
\]
Applying the Leibniz rule to
$f=B_j(1-u_j)^{-M}$ now gives
\[
 \frac{\mathcal E^qf}{n_j^qf}=H_{M,q}(u_j)+o(1)=O(1)
 \qquad(1\leq q\leq n)
\]
uniformly on $|z|=R_j$.
Since
\[
 z^n\frac{\dd^n}{\dd z^n}
 =\mathcal E(\mathcal E-1)\cdots(\mathcal E-n+1),
\]
it follows that $\mathcal E^qf=O(n_j^qf)=o(n_j^nf)$ for $q<n$.
Thus every lower-degree term in the displayed operator product is
$o(n_j^nf)$.  Therefore
\begin{equation}\label{eq:nth-sharpness-derivative-circle}
 \frac{z^nf^{(n)}(z)}{n_j^nf(z)}
 =H_{M,n}(u_j(z))+o(1),
 \qquad
 \left|\frac{f^{(n)}(z)}{f(z)}\right|\asymp n_j^n
 \quad(|z|=R_j).
\end{equation}

Equations \eqref{eq:nth-sharpness-P-circle} and
\eqref{eq:nth-sharpness-derivative-circle} imply
\[
 \log|f^{(n)}(z)|
 =-\bigl(M-n(\rho+1)\bigr)L_j+o(L_j)<0
\]
uniformly on $|z|=R_j$.
Thus $m(R_j,\infty,f^{(n)})=0$ for all sufficiently large $j$.
Every pole of $f$ has order $M$, and every pole of $f^{(n)}$ has order
$M+n$.  It follows that
\[
 T(R_j,f^{(n)})=(M+n)A_j.
\]
Jensen's formula for $P$ and
\eqref{eq:nth-sharpness-derivative-circle} also give
\[
 m(R_j,0,f^{(n)})=MA_j-n\log n_j+O(1).
\]
Consequently,
\[
 \delta(0,f^{(n)})
 \leq\frac{M-n(\rho+1)}{M+n}.
\]

For the stated space $V$, one has $D_V=D^n$.  The function $f^{(n)}$ has a pole at the zero $a$ of $P$, so its characteristic is positive for all
$r>a$.  Thus, applying \thmref{thm:rational-target-extension} with
$A=\{0\}$ gives
\[
\begin{aligned}
 1=\delta(0,f)
 &\leq\delta(0,f^{(n)})\left(1+\frac nM\right)
       +\frac{n(\rho+1)}M\\
 &\leq\frac{M-n(\rho+1)}{M+n}\left(1+\frac nM\right)
       +\frac{n(\rho+1)}M=1.
\end{aligned}
\]
Both inequalities are equalities, which proves the proposition.
Moreover,
\[
 \delta(0,f^{(n)})\left(1+\frac nM\right)+\frac nM
 =1-\frac{n\rho}{M}<1.
\]
Thus replacing $\rho+1$ by $1$ would make the estimate false.
\end{proof}

\section{Deficiency sums over finite-dimensional subspaces of
\texorpdfstring{$\C(z)$}{C(z)}}
\label{sec:rational-targets}

For a rational target $a\in\C(z)$ and a meromorphic function $f$ in $\D$
with $f\not\equiv a$, we extend the notation by setting
\[
 m(r,a,f)=m\!\left(r,\frac1{f-a}\right).
\]
If, in addition, $T(r,f)>0$ for all sufficiently large $r$, define
\[
 \delta(a,f)=\liminf_{r\to1^-}\frac{m(r,a,f)}{T(r,f)}.
\]
Let $V\subset\C(z)$ be an $n$-dimensional complex vector space with $n\geq1$
and choose a basis $v_1,\ldots,v_n$.  Define
\begin{equation}\label{eq:DV-definition}
 D_Vg=
 \frac{W(v_1,\ldots,v_n,g)}{W(v_1,\ldots,v_n)}.
\end{equation}

\begin{lemma}\label{lem:Wronskian-criterion}
Meromorphic functions $u_1,\ldots,u_m$ on the connected domain $\D$ are
linearly dependent over $\C$ if and only if
\[
 W(u_1,\ldots,u_m)\equiv0.
\]
\end{lemma}

This is the standard Wronskian criterion.  See
\cite[Part~VII, \S5, Problem~60]{PolyaSzego1998}.  The meromorphic formulation above follows by localization and analytic continuation.

\begin{proposition}\label{prop:DV-kernel}
The denominator in \eqref{eq:DV-definition} is not identically zero, the
operator $D_V$ is independent of the chosen basis, and
\[
 D_V=D^n+c_{n-1}D^{n-1}+\cdots+c_0,
 \qquad c_j\in\C(z).
\]
Moreover,
\[
 \ker\bigl(D_V\colon\mathcal M(\D)\to\mathcal M(\D)\bigr)=V,
\]
where $\mathcal M(\D)$ denotes the field of meromorphic functions in $\D$.
\end{proposition}

\begin{proof}
The functions $v_1,\ldots,v_n$ are linearly independent over $\C$, so
\lemref{lem:Wronskian-criterion} shows that their Wronskian is not identically
zero.  A change of basis multiplies both
Wronskians in \eqref{eq:DV-definition} by the same nonzero constant.  Expanding the numerator along its last column shows that $D_V$ is monic of order $n$
and has rational coefficients.

Clearly $V\subseteq\ker D_V$.  Conversely, if $D_Vg\equiv0$, then
$W(v_1,\ldots,v_n,g)\equiv0$.  Lemma~\ref{lem:Wronskian-criterion} gives constants
$c_1,\ldots,c_n,c$, not all zero, such that
\[
 c g+\sum_{j=1}^n c_jv_j=0.
\]
The independence of $v_1,\ldots,v_n$ forces $c\ne0$, and hence $g\in V$.
\end{proof}

\begin{lemma}\label{lem:simultaneous-finite-order}
Let $n\geq1$ and $q\geq1$, and let
\[
 L=D^n+c_{n-1}D^{n-1}+\cdots+c_0,
 \qquad c_j\in\C(z),
\]
and let $g_1,\ldots,g_q$ be meromorphic functions in $\D$, none identically zero, of order at most $\sigma<\infty$.  For every $\eps>0$, as $r\to1^-$,
\begin{align}
 &\frac1{2\pi}\int_0^{2\pi}
 \log^+\max_{1\leq\nu\leq q}
 \left|\frac{Lg_\nu}{g_\nu}(re^{i\theta})\right|\dd\theta \notag\\
 &\qquad\leq
 n(\sigma+1+\eps)\log\frac1{1-r}+O_{L,\boldsymbol g,\eps}(1).
 \label{eq:simultaneous-finite-order}
\end{align}
\end{lemma}

\begin{proof}
Fix $\eps>0$, put $\xi=\eps/2$, and set $\gamma=1/2$.  We first verify that differentiation does not increase the order.  Let $h\not\equiv0$ be a meromorphic function of finite order in $\D$, and suppose that
$h'\not\equiv0$.  Fix $r_0=1/2$.  For $r_0<r<1$, put $p=(1+r)/2$.  A pole of
$h$ of order $s$ is a pole of $h'$ of order $s+1$, and $h'$ has no other
poles.  Hence
\[
 N(r,\infty,h')=N(r,\infty,h)+\Nbar(r,\infty,h).
\]
The pointwise inequality
\[
 \log^+|h'|\leq\log^+|h|+\log^+\left|\frac{h'}h\right|
\]
holds away from the zeros and poles of $h$, and therefore almost everywhere on every circle.  It follows from \corref{cor:GG-log} and the monotonicity of
$T(r,h)$ that
\begin{align*}
 T(r,h')
 &\leq T(r,h)+\Nbar(r,\infty,h)
       +m\!\left(r,\frac{h'}h\right)\\
 &\leq 2T(p,h)+\log^+T(p,h)
       +\log\frac{p}{r(p-r)}+O_h(1).
\end{align*}
In the second inequality we used
\[
 \Nbar(r,\infty,h)\leq N(r,\infty,h)+O_h(1)
 \leq T(r,h)+O_h(1).
\]
Here and below, bounded terms are considered only for $r$ sufficiently close
to $1$.  Since
\[
 1-p=\frac{1-r}{2},
 \qquad
 \frac{p}{r(p-r)}=\frac{1+r}{r(1-r)},
\]
the definition of order shows that, for every $\eta>0$,
\[
 T(p,h)=O_{h,\eta}\!\left(
 (1-r)^{-(\rho(h)+\eta)}\right).
\]
Since $\rho(h)+\eta>0$, both $\log^+T(p,h)$ and
$\log(p/[r(p-r)])$ are also
$O_{h,\eta}((1-r)^{-(\rho(h)+\eta)})$.
The preceding estimate for $T(r,h')$ then gives
\[
 T(r,h')=O_{h,\eta}\!\left(
 (1-r)^{-(\rho(h)+\eta)}\right).
\]
Letting $\eta\downarrow0$ yields
\[
 \rho(h')\leq\rho(h).
\]
By induction,
\begin{equation}\label{eq:derivatives-preserve-order}
 \rho\bigl(g_\nu^{(k)}\bigr)\leq\sigma
 \quad\text{whenever }g_\nu^{(k)}\not\equiv0.
\end{equation}

Let $I$ be the set of indices $\nu$ for which $g_\nu$ is not rational.  If
$\nu\notin I$, then $Lg_\nu/g_\nu$ is rational.  Since the characteristic of a rational function is bounded as $r\to1^-$,
\[
 m\!\left(r,\frac{Lg_\nu}{g_\nu}\right)=O_{L,g_\nu}(1).
\]
Since there are only finitely many indices, the rational ones together
contribute $O_{L,\boldsymbol g}(1)$ to the radial mean in
\eqref{eq:simultaneous-finite-order}.  If $I=\varnothing$, the lemma follows.
We may therefore assume that $I\ne\varnothing$.  For $\nu\in I$, every
derivative
$g_\nu^{(k)}$ is not identically zero.  Indeed, if
$g_\nu^{(k)}\equiv0$ for some $k\geq1$, then $g_\nu$ is a polynomial and hence rational, contrary to the definition of $I$.  We may therefore apply
\thmref{thm:GG-source}, with the fixed values $r_0=1/2$ and $\gamma=1/2$, to
\[
 h=g_\nu^{(k)},
 \qquad \nu\in I,
 \qquad 0\leq k\leq n-1.
\]
The family of these functions is finite.  By
\eqref{eq:derivatives-preserve-order} and the definition of order, uniformly
over this family,
\[
 T\bigl(p,g_\nu^{(k)}\bigr)
 =O_{\boldsymbol g,n,\eps}\!\left(
 (1-r)^{-(\sigma+\xi)}\right).
\]
Moreover,
\[
 \frac{p}{r(p-r)}=O\!\left(\frac1{1-r}\right).
\]
The constants $\beta_1(g_\nu^{(k)},r_0)$ in
\thmref{thm:GG-source} are fixed, and there are only finitely many of them.
Consequently, simultaneously for $\nu\in I$ and $0\leq k\leq n-1$,
\begin{equation}\label{eq:successive-log-derivatives}
 \frac1{2\pi}\int_0^{2\pi}
 \left|
 \frac{g_\nu^{(k+1)}(re^{i\theta})}
      {g_\nu^{(k)}(re^{i\theta})}
 \right|^\gamma\dd\theta
 =O_{\boldsymbol g,n,\eps}\!\left(
 (1-r)^{-\gamma(\sigma+1+\xi)}\right).
\end{equation}

For $1\leq j\leq n$, the identity
\[
 \frac{g_\nu^{(j)}}{g_\nu}
 =\prod_{k=0}^{j-1}
 \frac{g_\nu^{(k+1)}}{g_\nu^{(k)}}
\]
holds away from the zeros and poles of the intermediate derivatives, and
hence holds as an identity of meromorphic functions.  It is therefore valid almost everywhere on every circle.  H\"older's inequality and
\eqref{eq:successive-log-derivatives} give
\begin{align}
 &\frac1{2\pi}\int_0^{2\pi}
 \left|
 \frac{g_\nu^{(j)}(re^{i\theta})}
      {g_\nu(re^{i\theta})}
 \right|^{\gamma/j}\dd\theta \notag\\
 &\quad\leq
 \prod_{k=0}^{j-1}
 \left\{
 \frac1{2\pi}\int_0^{2\pi}
 \left|
 \frac{g_\nu^{(k+1)}(re^{i\theta})}
      {g_\nu^{(k)}(re^{i\theta})}
 \right|^\gamma\dd\theta
 \right\}^{1/j} \notag\\
 &\quad=O_{\boldsymbol g,n,\eps}\!\left(
 (1-r)^{-\gamma(\sigma+1+\xi)}\right).
 \label{eq:generalized-log-derivatives}
\end{align}

Put
\[
 Y(re^{i\theta})=
 \max_{\substack{\nu\in I\\1\leq j\leq n}}
 \left|
 \frac{g_\nu^{(j)}}{g_\nu}(re^{i\theta})
 \right|^{1/j}.
\]
Hence \eqref{eq:generalized-log-derivatives} implies
\[
 \frac1{2\pi}\int_0^{2\pi}Y(re^{i\theta})^\gamma\dd\theta
 =O_{\boldsymbol g,n,\eps}\!\left(
 (1-r)^{-\gamma(\sigma+1+\xi)}\right).
\]
Since
\[
 \log^+x\leq\frac1\gamma\log(1+x^\gamma)
 \qquad(x\geq0),
\]
Jensen's inequality for the concave function $\log(1+x)$ gives
\begin{align*}
 \frac1{2\pi}\int_0^{2\pi}\log^+Y(re^{i\theta})\dd\theta
 &\leq\frac1\gamma\log\left(
 1+\frac1{2\pi}\int_0^{2\pi}Y(re^{i\theta})^\gamma\dd\theta
 \right)\\
 &\leq(\sigma+1+\xi)\log\frac1{1-r}
 +O_{\boldsymbol g,n,\eps}(1).
\end{align*}

For $\nu\in I$, the definition of $Y$ gives
\[
 \left|\frac{g_\nu^{(j)}}{g_\nu}\right|
 \leq1+Y^n,
 \qquad 0\leq j\leq n.
\]
Therefore
\[
 \left|\frac{Lg_\nu}{g_\nu}\right|
 \leq
 \left(1+\sum_{j=0}^{n-1}|c_j|\right)(1+Y^n).
\]
Because the coefficients $c_j$ are rational,
\[
 \frac1{2\pi}\int_0^{2\pi}
 \log^+\left(1+\sum_{j=0}^{n-1}
 |c_j(re^{i\theta})|\right)\dd\theta
 =O_L(1).
\]
Also,
\[
 \log(1+Y^n)\leq\log2+n\log^+Y.
\]
Combining these estimates with the bounded contribution of the rational
indices yields
\[
 \frac1{2\pi}\int_0^{2\pi}
 \log^+\max_{1\leq\nu\leq q}
 \left|\frac{Lg_\nu}{g_\nu}(re^{i\theta})\right|\dd\theta
 \leq n(\sigma+1+\xi)\log\frac1{1-r}
 +O_{L,\boldsymbol g,\eps}(1).
\]
Since $\xi=\eps/2<\eps$, the desired estimate follows.
\end{proof}

\begin{proposition}\label{prop:rational-target-estimates}
Let $f$ be meromorphic in $\D$ and of finite order $\rho$.  Assume that
$f\notin V$, let $q\geq1$, and let
$A=\{a_1,\ldots,a_q\}\subset V$ be a finite set of distinct functions.  Put
\[
 M_A(r)=\sum_{\nu=1}^q m(r,a_\nu,f).
\]
For every $\eps>0$, as $r\to1^-$,
\begin{equation}\label{eq:rational-target-first}
 M_A(r)
 \leq m(r,0,D_Vf)
 +n(\rho+1+\eps)\log\frac1{1-r}+O_{A,V,f,\eps}(1),
\end{equation}
and
\begin{align}
 M_A(r)
 &\leq T(r,f)+n\Nbar(r,\infty,f)-N(r,0,D_Vf) \notag\\
 &\quad+n(\rho+1+\eps)\log\frac1{1-r}+O_{A,V,f,\eps}(1).
 \label{eq:rational-target-second}
\end{align}
\end{proposition}

\begin{proof}
The condition $f\notin V$ and \propref{prop:DV-kernel} ensure that
$D_Vf\not\equiv0$, so every expression below is defined.  Set
$g_\nu=f-a_\nu$.  At a point at which all quantities are finite, choose
an index $\nu$ for which $|g_\nu|$ is minimal.  For $\mu\ne\nu$,
\[
 |a_\mu-a_\nu|=|g_\nu-g_\mu|\leq2|g_\mu|.
\]
Consequently,
\begin{equation}\label{eq:target-separation}
 \sum_{\mu=1}^q\log^+\frac1{|g_\mu|}
 \leq
 \log^+\frac1{|g_\nu|}
 +\sum_{1\leq i<j\leq q}\log^+\frac2{|a_i-a_j|}.
\end{equation}
For a rational function $R$, one has $T(r,R)=O_R(1)$ as $r\to1^-$.  Since each $a_i-a_j$ is a nonzero rational function, the radial mean of the last
sum is therefore $O_A(1)$.  By \propref{prop:DV-kernel}, $D_Va_\nu=0$, and hence
$D_Vg_\nu=D_Vf$.  Therefore
\[
 \log^+\frac1{|g_\nu|}
 \leq
 \log^+\frac1{|D_Vf|}
 +\log^+\left|\frac{D_Vg_\nu}{g_\nu}\right|.
\]
Each $g_\nu$ has order at most $\rho$.  Integrating
\eqref{eq:target-separation} and applying
\lemref{lem:simultaneous-finite-order} proves
\eqref{eq:rational-target-first}.

For the second estimate, retain the same pointwise choice of $\nu$.  The
identity
\[
 \log^+\frac1{|g_\nu|}
 =\log^+|g_\nu|
 +\log\left|\frac{D_Vf}{g_\nu}\right|-\log|D_Vf|
\]
and $D_Vf=D_Vg_\nu$ imply
\[
 \log^+\frac1{|g_\nu|}
 \leq
 \log^+|f|
 +\log^+\left|\frac{D_Vg_\nu}{g_\nu}\right|
 -\log|D_Vf|+\log2+\log^+|a_\nu|.
\]
Here we used
$\log^+|f-a_\nu|\leq\log^+|f|+\log^+|a_\nu|+\log2$.
The last two terms have radial mean $O_A(1)$.  Combining this with
\eqref{eq:target-separation}, integrating, applying
\lemref{lem:simultaneous-finite-order}, and using Jensen's formula for
$D_Vf$, we obtain
\begin{align}
 M_A(r)
 &\leq m(r,\infty,f)+N(r,\infty,D_Vf)-N(r,0,D_Vf) \notag\\
 &\quad+n(\rho+1+\eps)\log\frac1{1-r}+O_{A,V,f,\eps}(1).
 \label{eq:before-pole-count}
\end{align}
Let $P_V$ be the finite set of poles in $\D$ of the rational coefficients of
$D_V$.  At a pole of $f$ of order $s$ outside $P_V$, every summand in $D_Vf$
has pole order at most $s+n$.  At the nonzero points of $P_V$, the additional pole multiplicities are bounded in terms of the fixed coefficients, and their integrated contribution is $O_V(1)$.  By separating the contributions at the origin, whose difference is $O_{V,f}(1)$ as $r\to1^-$, and summing the local
multiplicity estimates over the nonzero poles, we obtain
\begin{equation}\label{eq:DV-pole-count}
 N(r,\infty,D_Vf)
 \leq N(r,\infty,f)+n\Nbar(r,\infty,f)+O_{V,f}(1).
\end{equation}
Substitution in \eqref{eq:before-pole-count}, together with the identity
\[
 T(r,f)=m(r,\infty,f)+N(r,\infty,f),
\]
proves \eqref{eq:rational-target-second}.
\end{proof}

\begin{proof}[Proof of Theorem~\ref{thm:rational-target-extension}]
Since $\lambda(f)>0$, we have $T(r,f)\to\infty$, and hence $f$ is not
rational.  In particular, $f\notin V$, so
\propref{prop:rational-target-estimates} applies.  The assertion is immediate when $A=\varnothing$.  We may therefore fix a nonempty finite set
$A=\{a_1,\ldots,a_q\}\subset V$, where $q\geq1$, and write
\[
 m_V(r)=m(r,0,D_Vf),
 \qquad
 N_V(r)=N(r,0,D_Vf),
 \qquad
 U_V(r)=T(r,D_Vf).
\]
The hypothesis on $T(r,D_Vf)$ makes $U_V(r)$ positive when $r$ is
sufficiently close to $1$.  For these radii, put
\[
 \alpha_V(r)=\min\left\{\frac{m_V(r)}{U_V(r)},1\right\}.
\]
By \propref{prop:rational-target-estimates}, for every $\eps>0$,
\begin{equation}\label{eq:DV-first-final}
 M_A(r)
 \leq m_V(r)+n(\rho+1+\eps)\log\frac1{1-r}
 +O_{A,V,f,\eps}(1).
\end{equation}
The same proposition also gives
\begin{equation}\label{eq:DV-second-final}
 M_A(r)
 \leq T(r,f)+n\Nbar(r,\infty,f)-N_V(r)
 +n(\rho+1+\eps)\log\frac1{1-r}
 +O_{A,V,f,\eps}(1).
\end{equation}
Multiply \eqref{eq:DV-first-final} by $1-\alpha_V(r)$ and
\eqref{eq:DV-second-final} by $\alpha_V(r)$, and add.  Jensen's formula gives
\[
 m_V(r)+N_V(r)=U_V(r)+O_{D_Vf}(1),
\]
and the same argument as in \lemref{lem:interpolation} yields
\begin{align}
 M_A(r)
 &\leq
 \alpha_V(r)\bigl(T(r,f)+n\Nbar(r,\infty,f)\bigr) \notag\\
 &\quad+n(\rho+1+\eps)\log\frac1{1-r}
 +O_{A,V,f,\eps}(1).
 \label{eq:DV-interpolated-final}
\end{align}
Choose $r_j\to1^-$ such that
\[
 \frac{m(r_j,0,D_Vf)}{T(r_j,D_Vf)}
 \longrightarrow\delta(0,D_Vf).
\]
Then
\[
 \alpha_V(r_j)\longrightarrow\min\{\delta(0,D_Vf),1\}
 \leq\delta(0,D_Vf).
\]
Divide \eqref{eq:DV-interpolated-final} by $T(r_j,f)$ and let $j\to\infty$.
As in the proof of \thmref{thm:main}, we have
\[
 \limsup_{r\to1^-}\frac{\Nbar(r,\infty,f)}{T(r,f)}=k(f),
 \qquad
 \limsup_{r\to1^-}\frac{\log(1/(1-r))}{T(r,f)}
 =\frac1{\lambda(f)}.
\]
Hence
\[
 \limsup_{j\to\infty}\frac{M_A(r_j)}{T(r_j,f)}
 \leq
 \delta(0,D_Vf)\bigl(1+n k(f)\bigr)
 +\frac{n(\rho+1+\eps)}{\lambda(f)}.
\]
Since
\[
 \sum_{a\in A}\delta(a,f)
 \leq\liminf_{r\to1^-}\frac{M_A(r)}{T(r,f)}
 \leq\limsup_{j\to\infty}\frac{M_A(r_j)}{T(r_j,f)},
\]
letting $\eps\downarrow0$ proves \eqref{eq:rational-target-extension} for the fixed finite set $A$.

\end{proof}

\section{Proof of Theorem~\ref{thm:valiron-rational-target-extension}}
\label{sec:valiron-proof}

\begin{proof}[Proof of Theorem~\ref{thm:valiron-rational-target-extension}]
Since $\alpha(f)>0$, the characteristic $T(r,f)$ tends to infinity as
$r\to1^-$.  Thus $f$ is not rational and, in particular, $f\notin V$.  The assertion is immediate for $A=\varnothing$, so fix
$A=\{a_1,\ldots,a_q\}\subset V$, with $q\geq1$, and set
\[
 g_\nu=f-a_\nu,
 \qquad
 \mathcal E_A(r)=\frac1{2\pi}\int_0^{2\pi}
 \log^+\max_{1\leq\nu\leq q}
 \left|\frac{D_Vg_\nu}{g_\nu}(re^{i\theta})\right|\dd\theta.
\]
Write $m_V(r)=m(r,0,D_Vf)$, $N_V(r)=N(r,0,D_Vf)$, and
$U_V(r)=T(r,D_Vf)$.  Since $D_Vg_\nu=D_Vf$, integration of
\eqref{eq:target-separation} gives
\[
 M_A(r)\leq m_V(r)+\mathcal E_A(r)+O_{A,V,f}(1).
\]
For the second estimate, use the pointwise inequality
\[
 \log^+\frac1{|g_\nu|}
 \leq\log^+|f|+
 \log^+\left|\frac{D_Vg_\nu}{g_\nu}\right|
 -\log|D_Vf|+\log2+\log^+|a_\nu|
\]
and integrate.  The terms involving $a_\nu$ and the nonzero rational
functions $a_i-a_j$ contribute $O_A(1)$.  Jensen's formula for $D_Vf$ and
\eqref{eq:DV-pole-count} give
\[
 M_A(r)\leq T(r,f)+n\Nbar(r,\infty,f)-N_V(r)
              +\mathcal E_A(r)+O_{A,V,f}(1).
\]
For all sufficiently large $r$, put
\[
 \beta_V(r)=\min\left\{\frac{m_V(r)}{U_V(r)},1\right\}.
\]
Multiply the first inequality by $1-\beta_V(r)$ and the second by
$\beta_V(r)$, and add.  Jensen's formula gives
\[
 m_V(r)+N_V(r)=U_V(r)+O_{D_Vf}(1).
\]
The calculation in \lemref{lem:interpolation} therefore gives
\begin{equation}\label{eq:valiron-interpolated}
 M_A(r)\leq
 \beta_V(r)\bigl(T(r,f)+n\Nbar(r,\infty,f)\bigr)
 +\mathcal E_A(r)+O_{A,V,f}(1).
\end{equation}

It remains to bound $\mathcal E_A(r)/T(r,f)$ along a sequence tending to $1$.
Suppose first that $\alpha(f)<+\infty$.  Then $f$, and hence every $g_\nu$,
has order zero.  Lemma~\ref{lem:simultaneous-finite-order}, with $\sigma=0$,
shows that
for every $\eps>0$,
\[
 \mathcal E_A(r)\leq n(1+\eps)\log\frac1{1-r}+O_{A,V,f,\eps}(1).
\]
Choose $r_j\to1^-$ so that
\[
 \frac{T(r_j,f)}{\log(1/(1-r_j))}\longrightarrow\alpha(f).
\]
It follows, after letting $\eps\downarrow0$, that
\begin{equation}\label{eq:valiron-error-sequence-finite}
 \limsup_{j\to\infty}\frac{\mathcal E_A(r_j)}{T(r_j,f)}
 \leq\frac{n}{\alpha(f)}.
\end{equation}

Now suppose that $\alpha(f)=+\infty$.  Since each $a_\nu$ is rational,
\[
 T(r,g_\nu)\leq T(r,f)+O_{A}(1).
\]
Let $\mathcal H$ be the finite family of functions $g_\nu^{(k)}$ with
$0\leq k<n$ for which $g_\nu$ is not rational.  Every member of
$\mathcal H$ is nonconstant.  Fix $r_0\in(0,1)$ and choose, for each
$h\in\mathcal H$, a constant $c_h\ne0$ such that
$T(r_0,c_hh)\geq e$.  Apply
\cite[Theorem~3.4.1\textup{(2)}]{CherryYe2001} with
$R=1$, $\phi(r)=1-r$, and $\psi(x)=x$.  Taking the union of the resulting exceptional sets gives a measurable set $E\subset(0,1)$ such that
\[
 \int_E\frac{\dd r}{1-r}<+\infty
\]
and, simultaneously for $h\in\mathcal H$,
\[
 m\!\left(r,\frac{h'}h\right)
 \leq2\log^+T(r,h)+\log\frac1{1-r}+O_h(1)
 \qquad(r\notin E).
\]
Here multiplication by $c_h$ leaves $h'/h$ unchanged and changes the
characteristic by at most a bounded term.
If $\mathcal H$ is empty, take $E=\varnothing$.

At every pole of a meromorphic function $h$, differentiation increases the pole order by one.  Together with
\[
 \log^+|h'|\leq\log^+|h|+
 \log^+\left|\frac{h'}h\right|,
\]
this gives
\[
 T(r,h')\leq2T(r,h)+m\!\left(r,\frac{h'}h\right)+O_h(1).
\]
Induction using this estimate shows that
\[
 T\bigl(r,g_\nu^{(k)}\bigr)
 =O_{A,f,n}\left(T(r,g_\nu)+\log\frac1{1-r}\right)
 \qquad(0\leq k\leq n,\ r\notin E)
\]
whenever $g_\nu$ is not rational.  Applying the same estimate once more and using
\[
 \frac{g_\nu^{(j)}}{g_\nu}
 =\prod_{k=0}^{j-1}
 \frac{g_\nu^{(k+1)}}{g_\nu^{(k)}}
 \qquad(1\leq j\leq n)
\]
therefore gives
\[
 m\!\left(r,\frac{g_\nu^{(j)}}{g_\nu}\right)
 =O_{A,f,n}\left(
 \log^+T(r,g_\nu)+\log\frac1{1-r}\right)
 \qquad(r\notin E).
\]
If $g_\nu$ is rational, then $D_Vg_\nu/g_\nu$ is rational and has bounded proximity.  Since the coefficients of $D_V$ are rational, the preceding
estimates and the finiteness of the family give
\begin{equation}\label{eq:valiron-exceptional-error}
 \mathcal E_A(r)=O_{A,V,f}\left(
  \log^+T(r,f)+\log\frac1{1-r}\right)
 \qquad(r\to1^-,\ r\notin E).
\end{equation}
Choose $s_j\to1^-$ with
\[
 \frac{T(s_j,f)}{\log(1/(1-s_j))}\longrightarrow+\infty,
\]
and define $t_j>s_j$ by
\[
 \log\frac1{1-t_j}=2\log\frac1{1-s_j}.
\]
The logarithmic measure of $[s_j,t_j]$ is
\[
 \int_{s_j}^{t_j}\frac{\dd r}{1-r}
 =\log\frac1{1-s_j}\longrightarrow+\infty.
\]
Hence, for all large $j$, we may choose
$r_j\in[s_j,t_j]\setminus E$.  Monotonicity of $T(r,f)$ then gives
\[
 \frac{T(r_j,f)}{\log(1/(1-r_j))}
 \geq\frac{T(s_j,f)}{2\log(1/(1-s_j))}\longrightarrow+\infty.
\]
Equation~\eqref{eq:valiron-exceptional-error} therefore implies
\begin{equation}\label{eq:valiron-error-sequence-infinite}
 \frac{\mathcal E_A(r_j)}{T(r_j,f)}\longrightarrow0
 =\frac{n}{\alpha(f)}.
\end{equation}

Whether $0<\alpha(f)<+\infty$ or $\alpha(f)=+\infty$, we have constructed a sequence $r_j\to1^-$ for which
\[
 \limsup_{j\to\infty}\frac{\mathcal E_A(r_j)}{T(r_j,f)}
 \leq\frac{n}{\alpha(f)}.
\]
Moreover, by the definition of Valiron deficiency,
\[
 \limsup_{j\to\infty}\beta_V(r_j)
 \leq\min\{\Delta(0,D_Vf),1\}\leq\Delta(0,D_Vf).
\]
Divide \eqref{eq:valiron-interpolated} by $T(r_j,f)$, use
\[
 \limsup_{r\to1^-}\frac{\Nbar(r,\infty,f)}{T(r,f)}=k(f),
\]
and let $j\to\infty$.  We obtain
\[
 \limsup_{j\to\infty}\frac{M_A(r_j)}{T(r_j,f)}
 \leq\Delta(0,D_Vf)\bigl(1+n k(f)\bigr)
       +\frac{n}{\alpha(f)}.
\]
Finally,
\[
 \sum_{a\in A}\delta(a,f)
 \leq\liminf_{r\to1^-}\frac{M_A(r)}{T(r,f)}
 \leq\limsup_{j\to\infty}\frac{M_A(r_j)}{T(r_j,f)},
\]
which proves \eqref{eq:valiron-rational-target-extension}.
\end{proof}

\begin{proposition}\label{prop:valiron-sharpness}
In \eqref{eq:improved-SheaSons-42}, the coefficient $n$ of $1/\alpha(f)$ is
best possible for every positive integer $n$, and the coefficient of the
Valiron-deficiency term cannot be decreased.  Neither term can be omitted.
\end{proposition}

\begin{proof}
Fix $n\geq1$.  We use a variant of the product construction in
\propref{prop:positive-order}.
Choose $\tau>0$ sufficiently large and set $b=\lceil\tau^2\rceil$ so that
\begin{equation}\label{eq:valiron-parameter-choice}
 \tau>n\log b.
\end{equation}
Set
\[
 N_j=b^j,
 \qquad r_j=e^{-\tau/N_j},
 \qquad a=e^{-\tau},
 \qquad C(w)=\frac{1-a^{-1}w}{1-aw},
\]
and define
\begin{equation}\label{eq:valiron-regular-product}
 P(z)=\prod_{j=1}^{\infty}C\bigl(z^{N_j}\bigr),
 \qquad f=P^{-1}.
\end{equation}
The product converges locally uniformly in $\D$.  Its zeros are the simple zeros of $1-a^{-1}z^{N_j}$, all of modulus $r_j$, and its denominators have
no zeros in $\D$.

Write $N_P(r)=N(r,0,P)$.  If $r_j\leq r<r_{j+1}$ and $t=-\log r$, then
\begin{gather}
 N_P(r)
 =\sum_{k=1}^jN_k\log\frac r{r_k}
   =j\tau-t\sum_{k=1}^jN_k
   =\tau j+O(1),                                      \label{eq:valiron-NP}\\
 \log\frac1{1-r}=j\log b+O(1).                       \label{eq:valiron-log-radius}
\end{gather}
Both error terms are uniform in the indicated annulus.

We record the estimate needed for the $n$th derivative.  Uniformly for
$r_j\leq r<r_{j+1}$,
\begin{align}
 \frac1{2\pi}\int_0^{2\pi}
 \left|\log|P(re^{i\theta})|-N_P(r)\right|\dd\theta
 &=O(1),                                               \label{eq:valiron-product-mean}\\
 \frac1{2\pi}\int_0^{2\pi}
 \left|\log\left|
 \frac{f^{(n)}(re^{i\theta})}{N_j^nf(re^{i\theta})}
 \right|\right|\dd\theta
 &=O(1).                                               \label{eq:valiron-derivative-mean}
\end{align}
Here and below the integrals are understood in the usual improper sense at the finitely many zeros or poles on a circle.

For completeness, we verify the part of these estimates not present in the proof of \propref{prop:nth-order-sharpness}.  Put
\[
 D_j=\sum_{k<j}N_k,
 \qquad w=z^{N_j},
 \qquad
 F(w)=\prod_{m=0}^{\infty}C\bigl(w^{b^m}\bigr),
\]
and, on $|z|=r$, put
\[
 E_j(z)=\prod_{k<j}
 \frac{1-az^{-N_k}}{1-az^{N_k}}.
\]
Then
\begin{equation}\label{eq:valiron-product-splitting}
 P(z)=(-a^{-1})^{j-1}z^{D_j}E_j(z)F(w).
\end{equation}
On these circles
\[
 a\leq|w|<a^{1/b},
 \qquad
 |az^{-N_k}|\leq a^{1-1/b}\quad(k<j).
\]
We first prove the absolute-mean estimate.  For every $q\geq0$,
\[
 \frac1{2\pi}\int_0^{2\pi}
 \left|\log|1-qe^{i\theta}|-\log^+q\right|\dd\theta
 \leq C_0.
\]
Indeed, when $0\leq q\leq1$, the mean of $\log|1-qe^{i\theta}|$ is zero and the positive part is bounded by $\log2$, so the positive and negative parts have equal bounded integrals.  When $q>1$, factor out $q$ and apply the first case to $q^{-1}$.

Put $t=-\log r$.  For $k<j$, both $ar^{-N_k}$ and $ar^{N_k}$ are at most
$q_0=ae^{\tau/b}<1$.  Since $tN_k\leq\tau/b$, the mean-value theorem gives,
pointwise in $\theta$,
\[
 \left|\log|1-ar^{-N_k}e^{iN_k\theta}|
       -\log|1-ar^{N_k}e^{iN_k\theta}|\right|
 \leq \frac{2ae^{\tau/b}}{1-q_0}\,tN_k.
\]
Consequently,
\[
 \frac1{2\pi}\int_0^{2\pi}
 \left|\log|E_j(re^{i\theta})|\right|\dd\theta
 \leq C_1tD_j=O(1),
\]
because $tD_j\leq\tau/(b-1)$.

Write $R=r^{N_j}$.  On $a\leq R<a^{1/b}$, the mean of
$\log|F(Re^{i\phi})|$ is $\log(R/a)$.  The estimate for
$\log|1-qe^{i\theta}|$ above controls the numerator and denominator of the
factors with indices $m=0$ and $m=1$.  For $m\geq2$,
\[
 a^{-1}R^{b^m}\leq a^{b^{m-1}-1},
 \qquad aR^{b^m}\leq a^{b^{m-1}+1}.
\]
Applying the inequality $\bigl|\log|1-u|\bigr|\leq |u|/(1-|u|)$ to the numerator and denominator and summing over $m\geq2$ gives
\[
 \frac1{2\pi}\int_0^{2\pi}
 \left|\log|F(Re^{i\phi})|-\log(R/a)\right|\dd\phi=O(1)
\]
uniformly in $R$.  The logarithms of the moduli of the constant and monomial factors in \eqref{eq:valiron-product-splitting} are independent of $\theta$,
and
\[
 (j-1)\tau+D_j\log r+\log(R/a)=N_P(r).
\]
Combining the bounds for $E_j$ and $F$ proves
\eqref{eq:valiron-product-mean}.

Let $\mathcal E=z\,\dd/\dd z$ and $\mathcal D=w\,\dd/\dd w$.  Direct
differentiation of $E_j$ gives, for every fixed $\ell\geq1$,
\begin{equation}\label{eq:valiron-Ej-derivatives}
 N_j^{-\ell}\left|\mathcal E^\ell\log E_j(z)\right|
 \leq C_\ell\frac{a^{1-1/b}}{b^\ell-1}.
\end{equation}
If $E_j$ is omitted from \eqref{eq:valiron-product-splitting}, the normalized expression $\mathcal E^nf/(N_j^nf)$ becomes
\begin{equation}\label{eq:valiron-model-derivative}
 K_{d_j,n}(w)
 =F(w)\bigl(\mathcal D-d_j\bigr)^nF(w)^{-1},
 \qquad d_j=\frac{D_j}{N_j}.
\end{equation}
Since $d_j\to1/(b-1)>0$, the function $K_{d_j,n}$ is not identically zero for all sufficiently large $j$.  Indeed,
$K_{d_j,n}(0)=(-d_j)^n\ne0$ for those $j$.  After multiplication by
\[
 (1-a^{-1}w)^n(1-a^{-1}w^b)^n,
\]
the functions in \eqref{eq:valiron-model-derivative} are analytic on a fixed neighbourhood of the closed annulus.  Jensen's formula therefore bounds the means over $|w|=R$ of both the positive and negative parts of
$\log|K_{d_j,n}|$, uniformly for large $j$.

We now restore $E_j$ and estimate $f^{(n)}$.  Divide $|z|=r$ into $N_j$
equal arcs.  On the $k$th arc put
\[
 z_{j,k}=re^{2\pi i k/N_j},\qquad
 \phi_{j,k}(x)=z_{j,k}\left(1+\frac{x}{N_j}\right).
\]
We show that the integral of the absolute logarithm on every rescaled arc is bounded by a constant independent of $j$, $k$, and $r$.

Suppose otherwise.  Choose sequences $j_\nu\to\infty$, $k_\nu$, and
$r_{j_\nu}\leq r_\nu<r_{j_\nu+1}$ for which the corresponding arc
integrals tend to infinity.  Write
\[
 N_\nu=N_{j_\nu},\qquad z_\nu=z_{j_\nu,k_\nu},\qquad
 \phi_\nu=\phi_{j_\nu,k_\nu},\qquad q_\nu=r_\nu^{N_\nu}.
\]
After taking a subsequence, $q_\nu\to q\in[a,a^{1/b}]$.  Put
$y=-\log q$.  Choose real numbers $L$ and $\varepsilon>0$ such that
\[
 y-b\tau<L<y-\tau,\qquad
 0<2\varepsilon<y-\frac{\tau}{b^2},
\]
and let $\Omega$ be the rectangle
\[
 L<\Real{x}<2\varepsilon,\qquad -1<\Imag{x}<2\pi+1.
\]
It contains $i[0,2\pi]$ and the point $x_0=y-\tau$.  For all sufficiently
large $\nu$, one has $\phi_\nu(\Omega)\subset\D$.  Since
\[
 N_\nu\log\left(1+\frac{x}{N_\nu}\right)=x+O(N_\nu^{-1})
\]
uniformly for $x\in\overline\Omega$, the only possible zeros of
$P\circ\phi_\nu$ in $\Omega$ are contributed by levels $j_\nu$ and
$j_\nu+1$, once $\nu$ is large.  The lower levels lie to the left of
$\Omega$, while every level $j_\nu+m$ with $m\geq2$ lies to its right.

Set
\[
 S_\nu(x)=\prod_{m=0}^{1}
 \left(1-a^{-1}\phi_\nu(x)^{b^mN_\nu}\right),
 \qquad G_\nu=f\circ\phi_\nu.
\]
The function $S_\nu G_\nu$ is holomorphic and has no zeros in $\Omega$.
Up to a nonzero constant, it has the exact factorization
\[
 S_\nu G_\nu
 =\phi_\nu^{-D_{j_\nu}}
 \prod_{m=1}^{j_\nu-1}
 \frac{1-a\phi_\nu^{N_\nu/b^m}}
      {1-a\phi_\nu^{-N_\nu/b^m}}
 \prod_{m=0}^{1}\left(1-a\phi_\nu^{b^mN_\nu}\right)
 \prod_{m=2}^{\infty}C\left(\phi_\nu^{b^mN_\nu}\right)^{-1}.
\]
For every compact $K\Subset\Omega$, logarithmic differentiation of the
lower blocks gives
\[
 \left|D_x\log
 \frac{1-a\phi_\nu^{N_\nu/b^m}}
      {1-a\phi_\nu^{-N_\nu/b^m}}
 \right|\leq C_Kb^{-m}\qquad(m\geq1),
\]
and logarithmic differentiation of the upper tail gives
\[
 u_{\nu,m}=\phi_\nu^{b^mN_\nu},\qquad
 a^{-1}|u_{\nu,m}(x)|\leq
 \exp\left(-c_K\tau b^{m-1}\right),\qquad
 \frac{u_{\nu,m}'(x)}{u_{\nu,m}(x)}
 =\frac{b^m}{1+x/N_\nu}
\]
for $m\geq2$, and hence
\[
 \left|D_x\log C\left(\phi_\nu^{b^mN_\nu}\right)^{-1}\right|
 \leq C_Kb^m\exp\left(-c_K\tau b^{m-1}\right)
 \qquad(m\geq2).
\]
The denominators in the first estimate stay uniformly away from zero on
$K$.  The second estimate follows from the positive distance between $K$
and the rescaled zeros of the levels $j_\nu+m$, $m\geq2$.

After passing to a diagonal subsequence, we may assume that
$e^{2\pi i k_\nu/b^m}$ converges for every fixed $m\geq1$.  The individual lower blocks in the displayed factorization then converge locally uniformly.
The upper blocks converge because $q_\nu\to q$, and
$D_{j_\nu}/N_\nu\to1/(b-1)$.  The two preceding estimates and the
Weierstrass $M$-test show that the logarithmic derivatives of the normalized functions
\[
 H_\nu=c_\nu S_\nu G_\nu,\qquad H_\nu(x_*)=1,
\]
converge locally uniformly on $\Omega$, where $x_*\in\Omega$ is fixed.
Since $\Omega$ is simply connected,
\[
 H_\nu(x)=\exp\left(
 \int_{x_*}^{x}\frac{H_\nu'(\zeta)}{H_\nu(\zeta)}\dd\zeta
 \right).
\]
It follows that $H_\nu\to H_\infty$ locally uniformly, where
$H_\infty$ is holomorphic and zero-free.

We also have
\[
 S_\nu\longrightarrow
 S_\infty(x)=(1-a^{-1}qe^x)\bigl(1-a^{-1}(qe^x)^b\bigr)
\]
locally uniformly.  Thus
$\widehat G_\nu:=c_\nu G_\nu=H_\nu/S_\nu$ converges meromorphically to
$G_\infty=H_\infty/S_\infty$.  The first factor of $S_\infty$ has a simple
zero at $x_0=\log(a/q)=y-\tau\in\Omega$, while
$H_\infty(x_0)\ne0$.  Therefore $G_\infty$ has a pole and
$G_\infty^{(n)}\not\equiv0$.

Define holomorphic functions on $\Omega$ by
\[
 \mathcal A_\nu=S_\nu^{n+1}\widehat G_\nu^{(n)},
 \qquad
 \mathcal B_\nu=S_\nu^nH_\nu.
\]
Since $S^{n+1}D^n(H/S)$ is a differential polynomial in the jets of $S$
and $H$, local uniform convergence and Cauchy's estimates show that
$\mathcal A_\nu$ and $\mathcal B_\nu$ converge locally uniformly.  Neither limit is identically zero.  Moreover,
\[
 \frac{\widehat G_\nu^{(n)}}{\widehat G_\nu}
 =\frac{\mathcal A_\nu}{\mathcal B_\nu}.
\]
We use the following elementary consequence of Jensen's formula.  If
$A_\nu$ and $B_\nu$ converge locally uniformly on a neighbourhood of
$i[0,2\pi]$ to holomorphic functions that are not identically zero, and
$N_\nu\to\infty$, then
\[
 \int_0^{2\pi}\left|\log\left|
 \frac{A_\nu(N_\nu(e^{it/N_\nu}-1))}
      {B_\nu(N_\nu(e^{it/N_\nu}-1))}
 \right|\right|\dd t=O(1).
\]
To verify this assertion, cover $i[0,2\pi]$ by finitely many discs whose
closures lie in the domain and whose enlarged boundaries contain no zero of the two limit functions.  Rouch\'e's theorem gives a uniform number of zeros on the enlarged discs.  Factoring those zeros leaves zero-free factors bounded above and below on the smaller discs.  The curves
$N(e^{it/N}-1)$ are uniformly bi-Lipschitz there, and the logarithm of each linear factor has a uniformly bounded integral.  This proves the assertion.

Apply it to $\mathcal A_\nu$ and $\mathcal B_\nu$.  The corresponding
integrals over the rescaled arcs are bounded along the extracted subsequence, contradicting their choice.  Hence the bound is uniform over all arcs.  On
the $k$th arc,
\[
 z=z_{j,k}e^{it/N_j}
   =\phi_{j,k}\bigl(N_j(e^{it/N_j}-1)\bigr),
 \qquad 0\leq t\leq2\pi,
\]
and the chain rule gives
\[
 \frac{f^{(n)}(z)}{N_j^nf(z)}
 =z_{j,k}^{-n}
 \frac{(f\circ\phi_{j,k})^{(n)}}{f\circ\phi_{j,k}}
 \bigl(N_j(e^{it/N_j}-1)\bigr).
\]
Since $\bigl|\log|z_{j,k}|\bigr|=O(N_j^{-1})$, each rescaled arc has a bounded integral of the absolute logarithm.  The change of variables $\dd\theta=\dd t/N_j$
and summation over the $N_j$ arcs prove
\eqref{eq:valiron-derivative-mean}.

Since $P(0)=1$, Jensen's formula and
\eqref{eq:valiron-product-mean} give
\[
 m(r,0,P)=O(1).
\]
Every pole of $f$ is simple.  Equations \eqref{eq:valiron-NP} and
\eqref{eq:valiron-log-radius} now yield
\begin{equation}\label{eq:valiron-example-data}
 T(r,f)=N_P(r)+O(1),
 \qquad
 \alpha(f)=\frac{\tau}{\log b}<+\infty,
 \qquad k(f)=1,
 \qquad \delta(0,f)=1.
\end{equation}

Equations \eqref{eq:valiron-product-mean} and
\eqref{eq:valiron-derivative-mean} show that
\[
 \frac1{2\pi}\int_0^{2\pi}\log|f^{(n)}(re^{i\theta})|\dd\theta
 =-N_P(r)+n\log N_j+O(1).
\]
The right side tends to $-\infty$ by
\eqref{eq:valiron-parameter-choice}.  Using the estimates
\eqref{eq:valiron-product-mean} and \eqref{eq:valiron-derivative-mean}, we
obtain $m(r,\infty,f^{(n)})=O(1)$.  Differentiation raises the order of every pole from $1$ to $n+1$.  Hence
\begin{align*}
 T(r,f^{(n)})&=(n+1)N_P(r)+O(1),\\
 m(r,0,f^{(n)})&=N_P(r)-n\log N_j+O(1).
\end{align*}
It follows that
\begin{equation}\label{eq:valiron-example-Delta}
 \Delta(0,f^{(n)})
 =\frac{\tau-n\log b}{(n+1)\tau}>0.
\end{equation}

The function $f$ omits $0$.  Corollary~\ref{cor:improved-SheaSons-42},
\eqref{eq:valiron-example-data}, and
\eqref{eq:valiron-example-Delta} give
\[
 1\leq\sum_{a\ne\infty}\delta(a,f)
 \leq\Delta(0,f^{(n)})(1+n k(f))+\frac n{\alpha(f)}
 =\frac{\tau-n\log b}{\tau}+\frac{n\log b}{\tau}=1.
\]
Thus equality holds.  Since both terms are positive, neither
coefficient can be decreased, and neither term can be omitted.
\end{proof}

\appendix

\section{Sharpness for arbitrary rational kernels}
\label{sec:kernelwise-sharpness}

We prove here that the estimates in
\thmref{thm:rational-target-extension} and
\thmref{thm:valiron-rational-target-extension} are attained for every
rational kernel, not only when $D_V=D^n$.

We first record two consequences of the structure of $D_V$.  Order a basis
$v_1,\ldots,v_n$ of $V$ and put
\[
 W_j=W(v_1,\ldots,v_j),\qquad W_0=1,\qquad
 \psi_j=\frac{W_j}{W_{j-1}},\qquad
 q_j=\frac{\psi_j'}{\psi_j}.
\]
The Wronskian definition of $D_V$ gives the factorization
\begin{equation}\label{eq:kernel-Polya-factorization}
 D_V=(D-q_n)\cdots(D-q_1).
\end{equation}
Each $q_j$ is the logarithmic derivative of a rational function, so its
finite poles are simple.  Write
\[
 D_V=D^n+\sum_{\ell=0}^{n-1}c_\ell D^\ell,
 \qquad c_\ell\in\C(z).
\]
Expanding the factorization and using the simple poles of the $q_j$ shows that $c_\ell$ has pole order at most $n-\ell$ at every finite point.
Consequently, for all $r$ sufficiently close to $1$,
\begin{equation}\label{eq:kernel-coefficient-bound}
 \max_{|z|=r}|c_\ell(z)|
 =O_V\bigl((1-r)^{-(n-\ell)}\bigr).
\end{equation}

\begin{proposition}\label{prop:all-V-nevanlinna-sharpness}
Let $n\geq1$, let $V\subset\C(z)$ be an $n$-dimensional complex vector space, let
$a_0\in V$, let $\rho>0$, and let $M$ be an integer satisfying
\[
 M>n(\rho+1).
\]
There is a meromorphic function $f$ in $\D$ such that
\begin{equation}\label{eq:all-V-positive-data}
 \rho(f)=\rho,\qquad
 \lambda(f)=M,\qquad
 k(f)=\frac1M,\qquad
 \delta(a_0,f)=1,
\end{equation}
and
\begin{equation}\label{eq:all-V-positive-deficiency}
 \delta(0,D_Vf)=\frac{M-n(\rho+1)}{M+n}>0.
\end{equation}
For $A=\{a_0\}$, equality holds in
\eqref{eq:rational-target-extension}.  In particular,
\begin{equation}\label{eq:all-V-positive-equality}
 1=\delta(0,D_Vf)\left(1+\frac nM\right)
   +\frac{n(\rho+1)}M.
\end{equation}
Both terms on the right are strictly positive.  The dependence on $\rho$
cannot be omitted.
\end{proposition}

\begin{proof}
Use the sequences, product $P$, and circles $|z|=R_j$ from the proof of
\propref{prop:nth-order-sharpness}, and write $h=P^{-M}$ and
$A_j=N(R_j,0,P)$.  The calculations there give
\begin{align}
 A_j&=L_j(1+o(1)),&
 \log n_j&=(\rho+1)L_j+o(L_j),&
 1-R_j&\asymp e^{-L_j},                              \label{eq:all-V-scales}\\
 \log|P(z)|&=A_j+O(j),&
 \left|\frac{h^{(\ell)}(z)}{h(z)}\right|
   &=O(n_j^\ell)\quad(0\leq\ell\leq n),&
 \left|\frac{h^{(n)}(z)}{h(z)}\right|&\asymp n_j^n
 \label{eq:all-V-derivatives}
\end{align}
uniformly on $|z|=R_j$.

By \eqref{eq:kernel-coefficient-bound},
\eqref{eq:all-V-scales}, and \eqref{eq:all-V-derivatives}, for
$0\leq\ell<n$,
\[
 \left|\frac{c_\ell h^{(\ell)}}{h^{(n)}}\right|
 =O\bigl((1-R_j)^{-(n-\ell)}n_j^{\ell-n}\bigr)
 =O\bigl(e^{-\rho(n-\ell)L_j+o(L_j)}\bigr)=o(1).
\]
Hence, uniformly on the selected circles,
\begin{equation}\label{eq:all-V-dominance}
 \frac{D_Vh}{h}=\frac{h^{(n)}}h\bigl(1+o(1)\bigr),
 \qquad
 \left|\frac{D_Vh}{h}\right|\asymp n_j^n.
\end{equation}

At all but finitely many zeros of $P$, $D_Vh$ has a pole of order $M+n$.
Therefore
\[
 N(R_j,\infty,D_Vh)=(M+n)A_j+O(1).
\]
Equations \eqref{eq:all-V-scales}--\eqref{eq:all-V-dominance} give
\[
 \log|D_Vh|
 =-\bigl(M-n(\rho+1)\bigr)L_j+o(L_j)<0
\]
uniformly on $|z|=R_j$.  Thus $m(R_j,\infty,D_Vh)=0$ for all large $j$,
and averaging over $|z|=R_j$ gives
\[
 m(R_j,0,D_Vh)=MA_j-n\log n_j+O(j).
\]
It follows that
\begin{equation}\label{eq:all-V-positive-upper}
 \delta(0,D_Vh)
 \leq\frac{M-n(\rho+1)}{M+n}.
\end{equation}

Set $f=a_0+h$.  Adding the rational function $a_0$ preserves $\rho$,
$\lambda$, and $k$, and gives $\delta(a_0,f)=1$.  Hence
\propref{prop:nth-order-sharpness} yields
\eqref{eq:all-V-positive-data}.  Since $a_0\in V=\ker D_V$, $D_Vf=D_Vh$.
The function $D_Vf$ has infinitely many poles, so its characteristic is
positive for all sufficiently large radii.  Apply
\thmref{thm:rational-target-extension} with $A=\{a_0\}$.  Since
$\delta(a_0,f)=1$, \eqref{eq:all-V-positive-upper} gives
\[
\begin{aligned}
 1
 &\leq\delta(0,D_Vf)\left(1+\frac nM\right)
       +\frac{n(\rho+1)}M\\
 &\leq
 \frac{M-n(\rho+1)}{M+n}\left(1+\frac nM\right)
 +\frac{n(\rho+1)}M=1.
\end{aligned}
\]
This proves \eqref{eq:all-V-positive-deficiency} and
\eqref{eq:all-V-positive-equality}.  Finally,
\[
 \delta(0,D_Vf)\left(1+\frac nM\right)+\frac nM
 =1-\frac{n\rho}{M}<1,
\]
which proves the last assertion.
\end{proof}

\begin{proposition}\label{prop:all-V-valiron-sharpness}
Let $n\geq1$, let $V\subset\C(z)$ be an $n$-dimensional complex vector
space, let $a_0\in V$, and let $M\geq1$ be an integer.  There is a number
$\tau_0=\tau_0(V,M,n)>n\log2/M$ such that, for every $\tau\geq\tau_0$,
there is a meromorphic function $f$ in $\D$ such that
\begin{equation}\label{eq:all-V-Valiron-data}
 \rho(f)=0,\qquad
 \lambda(f)=\alpha(f)=\frac{M\tau}{\log2},\qquad
 k(f)=\frac1M,\qquad
 \delta(a_0,f)=1,
\end{equation}
and
\begin{equation}\label{eq:all-V-Valiron-deficiency}
 \delta(0,D_Vf)=\Delta(0,D_Vf)
 =\frac{M\tau-n\log2}{(M+n)\tau}>0.
\end{equation}
For $A=\{a_0\}$, equality holds in both
\eqref{eq:rational-target-extension} and
\eqref{eq:valiron-rational-target-extension}.  In particular,
\begin{equation}\label{eq:all-V-Valiron-equality}
 1=\Delta(0,D_Vf)\left(1+\frac nM\right)
   +\frac n{\alpha(f)}.
\end{equation}
Both terms on the right are strictly positive.  The term $n/\alpha(f)$
cannot be omitted.
\end{proposition}

\begin{proof}
Put
\[
 a=e^{-\tau},\qquad N_j=2^j,\qquad
 r_j=e^{-\tau/N_j},\qquad
 C(w)=\frac{1-a^{-1}w}{1-aw},
\]
and define
\begin{equation}\label{eq:all-V-regular-product}
 P(z)=\prod_{j=1}^{\infty}C(z^{N_j}),
 \qquad h=P^{-M}.
\end{equation}
Write $N_P(r)=N(r,0,P)$.  Equations \eqref{eq:valiron-NP},
\eqref{eq:valiron-log-radius}, and \eqref{eq:valiron-product-mean} give,
uniformly for $r_j\leq r<r_{j+1}$,
\begin{gather}
 N_P(r)=\tau j+O_\tau(1),\qquad
 \log\frac1{1-r}=j\log2+O_\tau(1),                  \label{eq:all-V-regular-growth}\\
 \frac1{2\pi}\int_0^{2\pi}
 \left|\log|P(re^{i\theta})|-N_P(r)\right|\dd\theta
 =O_\tau(1).                                         \label{eq:all-V-P-mean}
\end{gather}
Consequently,
\begin{equation}\label{eq:all-V-h-data}
 \rho(h)=0,\qquad
 \lambda(h)=\alpha(h)=\frac{M\tau}{\log2},\qquad
 k(h)=\frac1M,\qquad
 \delta(0,h)=1.
\end{equation}
Since $h$ is transcendental and $\ker D_V=V\subset\C(z)$,
$D_Vh\not\equiv0$.

It remains to establish the uniform estimate
\begin{equation}\label{eq:all-V-DV-log-mean}
 \frac1{2\pi}\int_0^{2\pi}
 \left|\log\left|
 \frac{D_Vh(re^{i\theta})}{N_j^nh(re^{i\theta})}
 \right|
 \right|\dd\theta=O_{V,M,n,\tau}(1)
\end{equation}
for $r_j\leq r<r_{j+1}$.  We now prove this estimate.
Integrals through zeros or poles are understood in the improper sense.

We use the following local estimate.  Suppose that $\Omega$ contains
a closed neighbourhood of $i[0,2\pi]$, $N_\nu\to+\infty$, and
$A_\nu,B_\nu$ are holomorphic in $\Omega$ and converge locally uniformly to functions that are not identically zero.  For
\[
 \gamma_N(t)=N(e^{it/N}-1),\qquad 0\leq t\leq2\pi,
\]
one has
\begin{equation}\label{eq:local-log-fact}
 \int_0^{2\pi}
 \left|\log\left|
 \frac{A_\nu(\gamma_{N_\nu}(t))}
      {B_\nu(\gamma_{N_\nu}(t))}
 \right|\right|\dd t=O(1).
\end{equation}
Indeed, cover $i[0,2\pi]$ by finitely many discs whose concentric
enlargements are compactly contained in $\Omega$ and whose boundaries
contain no zero of the two limit functions.  Hurwitz's theorem and Jensen's formula factor $A_\nu$ and $B_\nu$, on the smaller discs, into a uniformly bounded number of linear factors and zero-free factors bounded above and
below.  The curves $\gamma_N$ are uniformly bi-Lipschitz and converge in
$C^1$ to $t\mapsto it$, so the logarithms of the linear factors have
uniformly bounded integrals.  This proves \eqref{eq:local-log-fact}.

Choose a nonzero polynomial $Q$ and an integer $d$ such that
\[
 U:=QV\subseteq\mathcal P_d,\qquad \deg Q\leq d,
\]
where $\mathcal P_d$ is the space of polynomials of degree at most $d$.
For a basis $p_1,\ldots,p_n$ of $U$, we have
\begin{equation}\label{eq:all-V-gauged-Wronskian}
 \frac{D_Vh}{h}
 =\frac{W(p_1,\ldots,p_n,Qh)}
 {W(p_1,\ldots,p_n)Qh}.
\end{equation}
Divide $|z|=r$ into $N_j$ equal arcs.  On the $k$th arc put
\[
 z_{j,k}=re^{2\pi i k/N_j},
 \qquad
 \phi_{j,k}(x)=z_{j,k}\left(1+\frac{x}{N_j}\right).
\]
Since $D_z=(N_j/z_{j,k})D_x$, equation
\eqref{eq:all-V-gauged-Wronskian} gives
\begin{equation}\label{eq:all-V-scaled-Wronskian}
 \left(\frac{z_{j,k}}{N_j}\right)^n
 \frac{D_Vh}{h}\bigl(\phi_{j,k}(x)\bigr)
 =\frac{W_x(P_{1,j,k},\ldots,P_{n,j,k},G_{j,k})}
 {W_x(P_{1,j,k},\ldots,P_{n,j,k})G_{j,k}},
\end{equation}
where $P_{\ell,j,k}=p_\ell\circ\phi_{j,k}$ and
$G_{j,k}=(Qh)\circ\phi_{j,k}$.

To obtain a bound independent of the arc, consider arbitrary sequences
$j_\nu\to\infty$, $k_\nu$, and
$r_{j_\nu}\leq r_\nu<r_{j_\nu+1}$.  Write
$N_\nu=N_{j_\nu}$ and $\phi_\nu=\phi_{j_\nu,k_\nu}$.  After passing to a
subsequence,
\[
 r_\nu^{N_\nu}\longrightarrow q\in[a,a^{1/2}].
\]
Put
\[
 U_\nu=\{p\circ\phi_\nu:p\in U\}.
\]
By compactness of the Grassmannian $\operatorname{Gr}(n,\mathcal P_d)$,
the sequence $(U_\nu)$ has a convergent subsequence.  Choose bases
$P_{1,\nu},\ldots,P_{n,\nu}$ converging coefficientwise to a basis
$P_1,\ldots,P_n$ of the limit space $U_\infty$.  After normalizing by
nonzero scalars and passing to a further subsequence, we may assume that
\[
 \widehat Q_\nu=\sigma_\nu Q\circ\phi_\nu
 \longrightarrow Q_\infty\not\equiv0
\]
in $\mathcal P_d$.  These changes do not alter the quotient in
\eqref{eq:all-V-scaled-Wronskian}.

For $\tau$ sufficiently large, take
\[
 \Omega=\left\{x:-\frac\tau2-2<\Real{x}<2,\ 
 |\Imag{x}|<2\pi(d+3)\right\}.
\]
For all large $\nu$, $\phi_\nu(\Omega)\subset\D$.  Define
\begin{equation}\label{eq:all-V-active-factor}
 S_\nu(x)
 =\left(1-a^{-1}\phi_\nu(x)^{N_\nu}\right)
  \left(1-a^{-1}\phi_\nu(x)^{2N_\nu}\right).
\end{equation}
The two factors have disjoint simple zeros, and they are exactly the
possible zero divisors of $P\circ\phi_\nu$ in $\Omega$ for all large
$\nu$.  Indeed, the rescaled real coordinate of a zero from level
$j_\nu+m$ is
\[
 -N_\nu\log r_\nu-\frac{\tau}{2^m}+o(1).
\]
The factors of $S_\nu^M$ cancel the poles of $h\circ\phi_\nu$ corresponding
to $m=0,1$.  The zeros from levels $j_\nu+m$ with $m\geq2$ lie to the right of $\Omega$, while those with $m\leq-1$ lie to the left.  Moreover,
\[
 S_\nu\longrightarrow
 S_\infty(x)=(1-a^{-1}qe^x)\left(1-a^{-1}(qe^x)^2\right)
\]
locally uniformly.

Put $H_\nu^0=S_\nu^M(h\circ\phi_\nu)$.  These functions are holomorphic and
zero-free in $\Omega$.  Choose a point $x_*$ outside the zero sets of
$S_\infty$ and $Q_\infty$, and choose $\kappa_\nu\ne0$ so that
$\kappa_\nu H_\nu^0(x_*)=1$.  Set
\[
 H_\nu=\kappa_\nu H_\nu^0,
 \qquad
 G_\nu=\kappa_\nu\widehat Q_\nu(h\circ\phi_\nu).
\]
This scalar normalization does not change the quotient in
\eqref{eq:all-V-scaled-Wronskian}.  Put $q_\nu=r_\nu^{N_\nu}$.  For
$-j_\nu+1\leq m<+\infty$, define
\[
 w_{\nu,m}(x)=\phi_\nu(x)^{2^{j_\nu+m}}.
\]
After this cancellation, $H_\nu^0$ has the exact decomposition
\[
 H_\nu^0
 =
 \prod_{p=1}^{j_\nu-1}C(w_{\nu,-p})^{-M}
 \prod_{\mu=0}^{1}(1-aw_{\nu,\mu})^M
 \prod_{m=2}^{\infty}C(w_{\nu,m})^{-M}.
\]
We also have, for $p\geq1$ and $m\geq0$, respectively,
\[
 \begin{aligned}
 w_{\nu,-p}(x)
 &=
 \eta_{\nu,p}q_\nu^{2^{-p}}
 \left(1+\frac{x}{N_\nu}\right)^{N_\nu2^{-p}},
 &\eta_{\nu,p}&=\exp\left(\frac{2\pi i k_\nu}{2^p}\right),\\
 w_{\nu,m}(x)
 &=
 q_\nu^{2^m}
 \left(1+\frac{x}{N_\nu}\right)^{N_\nu2^m}.
 \end{aligned}
\]
After passing to a diagonal subsequence, we may assume that $\eta_{\nu,p}$
converges for every fixed $p$.  Thus every fixed factor in this product
converges locally uniformly together with all its derivatives.

We next control the two infinite tails.  Choose $\tau_0$ sufficiently large
so that, for every $\tau\geq\tau_0$, every compact $K\Subset\Omega$, and all
sufficiently large $\nu$, the defining inequalities for $\Omega$ give
\[
 \begin{aligned}
 a^{-1}|w_{\nu,-p}(x)|&\geq e^{\tau/8},
 &a|w_{\nu,-p}(x)|&\leq e^{-\tau/2},
 &&1\leq p\leq j_\nu-1,\\
 a^{-1}|w_{\nu,m}(x)|&\leq e^{-\tau2^m/8},
 &&&&m\geq2,
 \end{aligned}
 \qquad x\in K.
\]
Moreover,
\[
 \frac{w_{\nu,m}'(x)}{w_{\nu,m}(x)}
 =\frac{2^m}{1+x/N_\nu}
\]
for every integer $m\geq-j_\nu+1$, and
\[
 \frac{\dd}{\dd x}\log\bigl(C(w)^{-M}\bigr)
 =
 Mw'\left(
 \frac{a^{-1}}{1-a^{-1}w}
 -\frac{a}{1-aw}
 \right).
\]
It follows that there is a constant $C_K$ such that
\begin{align*}
 \left|
 \frac{\dd}{\dd x}\log C(w_{\nu,-p})^{-M}
 \right|
 &\leq C_K2^{-p},
 &&1\leq p\leq j_\nu-1,\\
 \left|
 \frac{\dd}{\dd x}\log C(w_{\nu,m})^{-M}
 \right|
 &\leq C_K2^m e^{-\tau2^m/8},
 &&m\geq2.
\end{align*}
The logarithmic derivatives of the two middle factors also converge
locally uniformly.

Since multiplication by $\kappa_\nu$ does not change a logarithmic
derivative, the two preceding estimates and the Weierstrass $M$-test show that
\[
 L_\nu:=\frac{H_\nu'}{H_\nu}
 \longrightarrow L_\infty
\]
locally uniformly in $\Omega$.  The rectangle $\Omega$ is simply connected and $H_\nu(x_*)=1$.  Hence
\[
 H_\nu(x)
 =\exp\left(\int_{x_*}^{x}L_\nu(\zeta)\dd\zeta\right)
 \longrightarrow
 H_\infty(x):=
 \exp\left(\int_{x_*}^{x}L_\infty(\zeta)\dd\zeta\right)
\]
locally uniformly in $\Omega$.  In particular,
$H_\infty(x_*)=1$ and $H_\infty$ is zero-free.  Cauchy's estimates also
give
\[
 H_\nu^{(s)}\longrightarrow H_\infty^{(s)}
\]
locally uniformly for every $s\geq0$.

The functions
\[
 \widetilde G_\nu=S_\nu^M G_\nu=\widehat Q_\nu H_\nu
\]
converge locally uniformly to
$\widetilde G_\infty=Q_\infty H_\infty\not\equiv0$.  Hence
$G_\infty=S_\infty^{-M}\widetilde G_\infty$ is the meromorphic limit of
$G_\nu$.  The first factor of $S_\infty$ vanishes at
\[
 x=\log(a/q)+2\pi i m,\qquad m\in\mathbb Z.
\]
The rectangle $\Omega$ contains more than $d$ of these points.  The other
circular factors and $H_\infty$ are nonzero there, while the polynomial
$Q_\infty$ can cancel at most $d$ of them.  Thus $G_\infty$ has a pole and does not belong to $U_\infty$.

Define holomorphic functions on $\Omega$ by
\begin{align*}
 \mathcal A_\nu
 &=S_\nu^{M+n}
 W_x(P_{1,\nu},\ldots,P_{n,\nu},G_\nu),\\
 \mathcal B_\nu
 &=S_\nu^n W_x(P_{1,\nu},\ldots,P_{n,\nu})
 \widetilde G_\nu.
\end{align*}
For $0\leq s\leq n$, the expression
\[
 S_\nu^{M+n}D_x^s
 \bigl(S_\nu^{-M}\widetilde G_\nu\bigr)
\]
is a holomorphic differential polynomial in the jets of $S_\nu$ and
$\widetilde G_\nu$.  Hence $\mathcal A_\nu$ and $\mathcal B_\nu$ converge
locally uniformly, and
\[
 \frac{W_x(P_{1,\nu},\ldots,P_{n,\nu},G_\nu)}
 {W_x(P_{1,\nu},\ldots,P_{n,\nu})G_\nu}
 =\frac{\mathcal A_\nu}{\mathcal B_\nu}.
\]
The Wronskian criterion, applied on $\Omega$, shows that the limit of
$\mathcal A_\nu$ is not identically zero because
$G_\infty\notin U_\infty$.  The limit of
$\mathcal B_\nu$ is not identically zero because the limiting polynomial
basis is independent and neither $S_\infty$ nor
$\widetilde G_\infty$ is the zero function.  The local estimate
\eqref{eq:local-log-fact} now bounds the integral of the absolute logarithm on the rescaled arc.

If \eqref{eq:all-V-DV-log-mean} had no uniform bound, we could choose a
sequence of arcs whose integrals tend to $+\infty$.  Applying the preceding
subsequence argument to this sequence would produce a subsequence with
bounded integrals, a contradiction.  On the $k$th arc the
original angular variable is represented by
\[
 x=\gamma_{N_j}(t)=N_j(e^{it/N_j}-1),
 \qquad 0\leq t\leq2\pi,
\]
and $\dd\theta=\dd t/N_j$.  Thus each arc contributes $O(1/N_j)$.
Summing over the $N_j$ arcs in
\eqref{eq:all-V-scaled-Wronskian}, and using
$\bigl|\log|z_{j,k}|\bigr|=O(1/N_j)$, proves
\eqref{eq:all-V-DV-log-mean}.

We now finish the calculation.  Equations
\eqref{eq:all-V-P-mean} and \eqref{eq:all-V-DV-log-mean} give
\begin{equation}\label{eq:all-V-DV-central}
 \frac1{2\pi}\int_0^{2\pi}
 \left|\log|D_Vh(re^{i\theta})|
 +M N_P(r)-n\log N_j
 \right|\dd\theta=O(1)
\end{equation}
uniformly for $r_j\leq r<r_{j+1}$.  By
\eqref{eq:all-V-regular-growth},
\[
 -M N_P(r)+n\log N_j
 =-j(M\tau-n\log2)+O(1)\longrightarrow-\infty.
\]
Therefore
\[
 m(r,\infty,D_Vh)=O(1).
\]
At all but finitely many zeros of $P$, $D_Vh$ has a pole of order $M+n$.
Hence
\[
 N(r,\infty,D_Vh)=(M+n)N_P(r)+O(1).
\]
Averaging the expression inside the absolute value in
\eqref{eq:all-V-DV-central} now gives
\[
 T(r,D_Vh)=(M+n)N_P(r)+O(1),
\qquad
 m(r,0,D_Vh)=M N_P(r)-n\log N_j+O(1).
\]
Together with \eqref{eq:all-V-regular-growth}, these formulas yield
\begin{equation}\label{eq:all-V-regular-deficiency}
 \delta(0,D_Vh)=\Delta(0,D_Vh)
 =\frac{M\tau-n\log2}{(M+n)\tau}.
\end{equation}

Set $f=a_0+h$.  Adding the rational function $a_0$ preserves $\rho$,
$\lambda$, $\alpha$, and $k$, and gives $\delta(a_0,f)=1$, while
$a_0\in V=\ker D_V$ gives $D_Vf=D_Vh$.  Hence
\eqref{eq:all-V-h-data} and \eqref{eq:all-V-regular-deficiency} prove
\eqref{eq:all-V-Valiron-data} and
\eqref{eq:all-V-Valiron-deficiency}.  The operator $D_Vf$ has infinitely
many poles and therefore has eventually positive characteristic.  Finally,
\[
 \Delta(0,D_Vf)\left(1+\frac nM\right)
 +\frac n{\alpha(f)}
 =\frac{M\tau-n\log2}{M\tau}
  +\frac{n\log2}{M\tau}=1.
\]
This proves equality in
\eqref{eq:valiron-rational-target-extension}.  Since $\rho(f)=0$ and
$\lambda(f)=\alpha(f)$, the same calculation proves equality in
\eqref{eq:rational-target-extension}.  Removing $n/\alpha(f)$ leaves the
strictly smaller value $1-n\log2/(M\tau)$.
\end{proof}


\begin{thebibliography}{99}

\bibitem{BrannanHayman1989}
D. A. Brannan and W. K. Hayman,
\emph{Research problems in complex analysis},
Bull. London Math. Soc. \textbf{21} (1989), no.~1, 1--35.

\bibitem{CherryYe2001}
W. Cherry and Z. Ye,
\emph{Nevanlinna's Theory of Value Distribution: The Second Main Theorem and Its Error Terms},
Springer Monographs in Mathematics, Springer, Berlin, 2001.

\bibitem{Ciechanowicz2016}
E. Ciechanowicz,
\emph{Defective functions of meromorphic functions in the unit disc},
J. Appl. Anal. \textbf{22} (2016), no.~1, 15--26.

\bibitem{FrankWeissenborn1986}
G. Frank and G. Weissenborn,
\emph{Rational deficient functions of meromorphic functions},
Bull. London Math. Soc. \textbf{18} (1986), no.~1, 29--33.

\bibitem{HaymanLingham2019}
W. K. Hayman and E. F. Lingham,
\emph{Research Problems in Function Theory: Fiftieth Anniversary Edition},
Problem Books in Mathematics, Springer, Cham, 2019.

\bibitem{Miles1992}
J. Miles,
\emph{A sharp form of the lemma on the logarithmic derivative},
J. London Math. Soc. (2) \textbf{45} (1992), no.~2, 243--254.

\bibitem{PolyaSzego1998}
G. P\'olya and G. Szeg\H{o},
\emph{Problems and Theorems in Analysis II: Theory of Functions, Zeros,
Polynomials, Determinants, Number Theory, Geometry},
translated by C. E. Billigheimer, Classics in Mathematics, Springer, Berlin,
1998, reprint of the 1976 edition.

\bibitem{SheaSons1986}
D. F. Shea and L. R. Sons,
\emph{Value distribution theory for meromorphic functions of slow growth in the disk},
Houston J. Math. \textbf{12} (1986), no.~2, 249--266.

\bibitem{Sons1983}
L. R. Sons,
\emph{Unbounded functions in the unit disc},
Internat. J. Math. Math. Sci. \textbf{6} (1983), no.~2, 201--242.

\bibitem{Steinmetz1986}
N. Steinmetz,
\emph{Eine Verallgemeinerung des zweiten Nevanlinnaschen Hauptsatzes},
J. Reine Angew. Math. \textbf{368} (1986), 134--141.

\bibitem{Tsuji1959}
M. Tsuji,
\emph{Potential Theory in Modern Function Theory},
Maruzen, Tokyo, 1959.

\bibitem{Ullrich1929}
E. Ullrich,
\emph{\"Uber die Ableitung einer meromorphen Funktion},
Sitzungsber. Preuss. Akad. Wiss. Phys.-Math. Kl. (1929), 592--608.

\bibitem{Yamanoi2004}
K. Yamanoi,
\emph{The second main theorem for small functions and related problems},
Acta Math. \textbf{192} (2004), no.~2, 225--294.

\bibitem{Yamanoi2005}
K. Yamanoi,
\emph{Defect relation for rational functions as targets},
Forum Math. \textbf{17} (2005), no.~2, 169--189.

\end{thebibliography}
\end{document}